\documentclass[11pt]{article}
\usepackage{amsmath, amssymb,amsthm}

\title{The Spectrum of Heavy Tailed Random Matrices}
\author{
G\'erard Ben Arous\thanks{Swiss Federal Institute of Technology
(EPFL), CH-1015 Lausanne, Switzerland and Courant Institute of
Mathematical Sciences, New York University, 251 Mercer Street, New
York NY 10012, E-mail: gba1@nyu.edu}, Alice Guionnet\thanks{Ecole
Normale Sup\'erieure de Lyon, Unit\'e de Math\'ematiques pures et
appliqu\'ees, UMR 5669,46 All\'ee d'Italie, 69364 Lyon Cedex 07,
France. E-mail: aguionne@umpa.ens-lyon.fr 
}}

\newtheorem{prop}{Proposition}[section]
\newtheorem{theo}[prop]{Theorem}
\newtheorem{defi}[prop]{Definition}
\newtheorem{lem}[prop]{Lemma}

\newtheorem{rmk}[prop]{Remark}
\def\P{\mathbb P}
\def\E{\mathbb E}

\def\E{\mathbb E}
\def\C{\mathbb C}
\def\e{\epsilon}

\def\N{{\mathbb N}}

\def\tr{{\mbox{tr}}}

\def\tr{{\mbox{tr}}}

\def\xx{$\square$}

\def\a{\alpha}

\def\d{\delta}
\def\e{\epsilon}

\def\l{\lambda}

\def\ra{\rightarrow}

\def\Ca{{\cal C}}

\def\Fa{{\cal F}}

\def\La{{\cal L}}

\def\mun{{\hat\mu}}

\def\lbk{\lbrack}
\def\rbk{\rbrack}

\def\part{\partial}

\def\ot{\otimes}

\def\ts{\times}
\def\Pa{{\mathcal{P}}}

\long\def\symbolfootnote[#1]#2{\begingroup
\def\thefootnote{\fnsymbol{footnote}}\footnote[#1]{#2}\endgroup}

\def\R{{\mathbb R}}

\begin{document}

\renewcommand{\refname}{\Large{References}}

\maketitle

\symbolfootnote[0]{{\it MSC 2000 subject classifications.}  primary 
 15A52,  60E07.} 
 
\symbolfootnote[0]{{\it Key words.} Random matrices, stable distributions}

\symbolfootnote[0]{This work was partially supported by Miller institute for Basic
 Research in Science, University of California Berkeley.}

\begin{abstract}
Let $X_N$ be an $N\ts N$ random symmetric matrix
with independent equidistributed entries
modulo the symmetry constraint.
If the law $P$ of the entries
has a finite second moment,
it was shown by Wigner \cite{wigner}
that the empirical distribution of the eigenvalues
of $X_N$, once renormalized by $\sqrt{N}$,
converges almost surely and in expectation  to
the so-called semicircular distribution as $N$ goes to infinity.
In this paper we study the same question  when
$P$ is in the domain of attraction
of an $\alpha$-stable law. We prove that
if we renormalize the eigenvalues by a constant  $a_N$
of order $N^{\frac{1}{\alpha}}$, the corresponding
spectral distribution converges in expectation towards
a law $\mu_\alpha$ which only depends
on $\alpha$. We characterize $\mu_\alpha$
and study some of its properties; it is a heavy-tailed
probability measure which is absolutely continuous
with respect to Lebesgue measure except possibly on a compact  set
of capacity zero.

\end{abstract}
\section{Introduction}
We study the asymptotic behavior of the spectral measure of large
random real symmetric matrices with independent identically
distributed heavy tailed entries. Let $(x_{ij}, 1\le i\le j< \infty)
$ be an infinite array of i.i.d real variables with common
distribution $P$ living in a probability
space $(\Omega,\P)$.
 Denote by $X_N$ the $N\ts N$ symmetric matrix given
by:
$$ X_N(i,j)=x_{ij}\,  \mbox{ if }\, i\le j,\, x_{ji}\, \mbox{ otherwise.}$$

If the entries have a finite second moment
$\sigma^2=\E[x_{ij}^2]=\int x^2 dP(x)$, and if
$(\lambda_1,\cdots,\lambda_N)$ are the eigenvalues of
$\frac{X_N}{\sqrt{N}}$ then Wigner's theorem (see \cite{wigner} and
generalizations in \cite{pastur,bai}) asserts that the empirical
spectral measure $\frac{1}{N}\sum_{i=1}^N \delta_{\lambda_i}$ of the
matrix $\frac{X_N}{\sqrt{N}}$ converges weakly almost surely to the
semi-circle distribution
$$\sigma(dx)=\frac{1}{2\pi \sigma^2}\sqrt{4\sigma^2 -x^2} dx.$$

We will consider here the case of heavy tailed entries, when the
second moment $\sigma^2$ is infinite. We will assume that the common
distribution of the absolute values of the $x_{ij}$'s is in the
domain of attraction of an $\alpha$-stable law, for $\alpha\in
]0,2[$, i.e that there exists a slowly varying function L such that
\begin{equation}
\P(|x_{ij}|\geq u)= \frac{L(u)}{u^{\alpha}}. \label{stabledomain}
\end{equation}
We introduce the normalizing constant $a_N$ by:
\begin{equation}
a_N= \inf(u, \P[ |x_{ij}| \geq u]\le \frac{1}{N}).
\label{normalisation}
\end{equation}
It is clear  that $a_N$ is roughly of order
$N^{\frac{1}{\alpha}}$, indeed there exists another slowly varying
function $L_0$ such that
\begin{equation}
\label{normalisation2}
a_N= {L_0(N)}{N^{\frac{1}{\alpha}}}.
\end{equation}
We then consider the matrix $A_N:= a_N^{-1} X_N$, its eigenvalues
$(\lambda_1,\cdots,\lambda_N)$, and its spectral measure
$\mun_{A_N}:=\frac{1}{N}\sum_{i=1}^N \delta_{\lambda_i}$. Our main
result is

\begin{theo}\label{main}
Let $\alpha\in ]0,2[$ and
assume \eqref{stabledomain}.
\begin{enumerate}
\item There exists a probability measure $\mu_{\alpha}$ on $\R$
such that the mean spectral measure $\E[\mun_{A_N}]$ converges
weakly to $\mu_{\alpha}$.
\item  $\mun_{A_N}$ converges weakly in
probability to $\mu_{\alpha}$. More precisely, for any bounded
continuous function $f$, $\int f(x)d\mun_{A_N}(x)$ converges in
probability to $\int f(x) d\mu_{\alpha}(x)$.
\item Let $(N_k)_{k \ge 1}$ be an increasing sequence of integers such that
$\sum_{k=1}^{\infty} N_k^{-\varepsilon}<\infty$ for some
$\varepsilon<1$, then the subsequence $\hat \mu_{A_{N_k}}$ converges
almost surely weakly to $\mu_{\alpha}$.
\end{enumerate}
\end{theo}

\begin{rmk}
\label{Remark1} We note that the hypothesis \eqref{stabledomain}
concerns only the tail behavior of the distribution of the absolute
values of the entries. We make no assumption about the skewness of
the distribution of the entries, i.e about their right or left
tails.
\end{rmk}
\begin{rmk}
It  would be useful to control better the fluctuations in Theorem
\ref{main} and establish almost sure convergence for the whole
sequence $\mun_{A_N}$.
\end{rmk}
Our approach is classical. It consists in proving the convergence of
the resolvent, i.e of the mean of the Stieltjes transform of the
spectral measure, by proving tightness and characterizing uniquely
the possible limit points. We first prove, in section \ref{cutoff},
that it is possible , for all later purposes, to truncate the large
values of the entries at appropriate levels. We then proceed, in
section \ref{sectiontight}, to show tightness for the spectral
measures of the truncated and original matrices $A_N$ . We then
introduce, in section \ref{induction}, the following important
quantity: for $z\in\C\backslash\R$, we define the probability
measure $L_N^z$ on $\C$ by
$$L_N^z=\frac{1}{N}\sum_{k=1}^N \delta_{(z-A_N)^{-1}_{kk}}$$
i.e the empirical measure of the diagonal elements of the resolvent
of $A_N$ at $z\in\C\backslash \R$. The classical Schur complement
formula is our basic linear algebraic tool to study $L_N^z$
recursively on the dimension, as is usual when the resolvent method
is used (see e.g \cite{pastur} or \cite{bai}). In section
\ref{limitingeq}, using an argument of concentration of measure and
borrowing classical techniques from the theory of triangular arrays
of i.i.d random variables, we show that the limit points $\mu^z$ of
$L_N^z$ satisfy a fixed point equation in the space of probability
measures on $\C$. Even though we cannot prove uniqueness of the
solution to this equation, we manage in section \ref{secuniq} to
prove the uniqueness of the solution to the resulting equation for
$\int x^{\frac{\alpha}{2}}d\mu^z(x)$, which in turn gives the
uniqueness of $\int x d\mu^z(x)$. This is enough to characterize
uniquely the limit points of $\E[\mun_{A_N}]$  and thus the
convergence of $\E[\mun_{A_N}]$ to $\mu_{\alpha}$.

Once the question of convergence is settled by Theorem \ref{main},
the next question is to describe the limiting measure
$\mu_{\alpha}$. We will discuss in this article three different
characterizations of $\mu_{\alpha}$. Our approach leads directly to
the following first characterization of $\mu_{\alpha}$ through its
Stieltjes transform, defined for $z\in \C\backslash\R$ by:

\begin{equation}
\label{Stieltjes} G_{\alpha}(z)=\int (z-x)^{-1} d\mu_{\alpha}(x).
\end{equation}

Define the entire function g on $\C$ by
\begin{equation}
\label{g} g_{\alpha}(y)= \frac{2}{\alpha} \int_0^\infty
e^{-v^{\frac{2}{\alpha}}} e^{-vy}dv
\end{equation}

We will also need the constants
$C(\alpha)=\frac{e^{i\frac{\pi\alpha}{2}}}{
\Gamma(\frac{\alpha}{2})}$ and
$c(\alpha)=\cos(\frac{\pi\alpha}{4})$.

\begin{theo}\label{main2}
\begin{enumerate}
\item  There exists a unique function $Y_z$, analytic on the half 
plane $\C^+=\{ z \in \C,
Imz>0\}$, tending to zero at infinity, and such that
$$ C(\alpha) g(c(\alpha)Y_z)= Y_z (-z)^{\alpha}$$
\item The probability measure $\mu_{\alpha}$ of Theorem \ref{main} is
uniquely described by its Stieltjes transform given, for $z\in
\C^+$, by
\begin{equation}
\label{eqGint} G_{\alpha}(z)=-\frac{1}{z}\int_0^\infty e^{-t}
e^{-c(\alpha) t^{\frac{\alpha}{2}} Y_z} dt
\end{equation}

\end{enumerate}

\end{theo}
\begin{rmk} Note that $\mu_\alpha$
depends continuously on $\alpha\in(0,2)$ since $Y_z$
as described above is continuous in $\alpha$, at list for sufficiently large
$z$, a remark which insures the continuity of $G_\alpha(z)$
at list for sufficiently large $z$ and therefore the continuity of $\mu_\alpha$.

\end{rmk}

Using the characterization given in Theorem \ref{main2}, we prove in
section \ref{studylimit} the following properties of $\mu_{\alpha}$.
\begin{theo}\label{main3} The probability
measure $\mu_{\alpha}$ of Theorem \ref{main} satisfies
\begin{enumerate}
\item $\mu_{\alpha}$ is symmetric.
\item $\mu_{\alpha}$ has unbounded support.
\item There exists a (possibly empty) compact subset of the real line 
$K_{\alpha}$  of capacity zero, such that the measure $\mu_{\alpha}$
has a  smooth density $\rho_{\alpha}$ on the open complement
$U_{\alpha}= \R \backslash K_{\alpha}$ .
\item  $\mu_{\alpha}$ has heavy tails. There exists
a constant $L_{\alpha}>0$ such that, when $|x|\ra\infty$
$$\rho_{\alpha}(x)\sim \frac{L_{\alpha}}{|x|^{\alpha+1}}$$
\end{enumerate}
\end{theo}

A second and different characterization of $\mu_{\alpha}$ is
proposed in the physics literature by Cizeau-Bouchaud \cite{CB}.
This description has been controversial (see \cite{burda} for a
discussion and numerical simulations). The strategy used in
\cite{CB} is also based on the convergence of the resolvent, but on
the real axis as opposed to our proof of convergence away from the
real axis. We unfortunately cannot make sense of the strategy used
in \cite{CB}.  We discuss in section \ref{sectCB} the link between
our characterization given in Theorem \ref{main2} and the
Bouchaud-Cizeau characterization (after correction of a small
typographical error in \cite{CB} already noted by \cite{burda}).

\begin{rmk}
We are not able to prove in general that the exceptional set $K_{\alpha}$ of
Theorem \ref{main2} is empty, or reduced to zero, even though we
conjecture this is true.  This would say that $\mu_{\alpha}$ has a
smooth density everywhere (except may be at zero) as suggested by
numerical simulations and accepted by the physics literature.
Recent work in progress with A. Dembo indicates 
that $K_\alpha$ is at most the
origin for $\alpha\le 1$.  This question is discussed
further in Section \ref{studylimit}.
\end{rmk}

\bigskip
We also describe below (in section \ref{zaka}) a third
characterization of $\mu_{\alpha}$, more combinatorial in nature. It
is based on an extension (due to I.Zakharevich, \cite{zakh} ) of the
classical moment method  rather than the resolvent approach used
both by \cite{CB} and us. Obviously because of the heavy tails and
thus of the absence of moments, one would have to do it first for
truncated matrices and then try to lift the truncation. More
precisely if one truncates the entries at the level $Ba_N$, for a
fixed $B>0$ and define $x_{ij}^B= x_{ij}1_{|x_{ij}|\le Ba_N }$ one
can compute the moments of the empirical measures $\mun_{A_N^B}$ of
the truncated matrix $A_N^B(ij)=a_N^{-1} x_{ij}^B$
$$\int x^k d\mun_{A_N^B}(x)=\frac{1}{N}\tr\left((A_N^B)^k\right).$$
and study their convergence when N tends to infinity.
We establish
in section \ref{zaka} that
\begin{theo}\label{main4}With the above notations,and under the
hypothesis of Theorem (\ref{main})and the additional hypothesis:
\begin{equation}
\label{skewness} \lim_{u \ra \infty}
\frac{\P(x(ij)>u)}{\P(|x(ij)|>u)}= \theta \in [0,1]
\end{equation}
\begin{enumerate}
\item $\E[\mun_{A_N^B}]$ converges weakly to a probability
measure $\mu^B_{\alpha}$ uniquely determined by its moments and
independent of the parameter $\theta$. This measure $\mu^B_{\alpha}$
has unbounded support and is symmetric.
\item $\mu^B_{\alpha}$ converges weakly to  $\mu_{\alpha}$
as $B$ tends to infinity.
\end{enumerate}
\end{theo}

The moments of $\mu^B_{\alpha}$ are described combinatorially in
Section \ref{zaka}. Thus Theorem \ref{main4} gives a third,
independent, description of the limiting measure $\mu_{\alpha}$. As
we will see in Section \ref{zaka}, the first part of Theorem
(\ref{main4}) is a direct consequence of a general combinatorial
result of I.Zakharevich and its proof is essentially given in
\cite{zakh}. The convergence of these Zakharevich measures to our
$\mu_{\alpha}$ establishes a link between this combinatorial
description and the one we have given in terms of Stieltjes
transforms in Theorem \ref{main2}. This link is far from
transparent.

\begin{rmk}
We note that the limiting measure $\mu^B_{\alpha}$ is in fact
independent of the skewness parameter $\theta$. Thus it is
insensitive to the hypothesis (\ref{skewness}) about the upper and
lower tails of the distribution of the entries. This is coherent
with Remark \ref{Remark1}.
\end{rmk}

\begin{rmk}
The case $\alpha=2$ is covered neither by the classical Wigner
theorem (which asks for a second moment) nor by our results so far.
In fact it is easy to see, using the combinatorial approach of
Theorem \ref{main4} that the limit law is then the semi-circle, even
though the normalization differs from the usual one.
\end{rmk}

Finally, let us mention that the behavior of the edge of the
spectrum of heavy tailed matrices (when $\alpha\in ]0,2[$) has been
established by Soshnikov \cite{soshi}. The largest eigenvalues are
asymptotically, in the scale $a_N^2$, distributed as a Poisson point
process with intensity $\alpha^{-1} x^{-\alpha-1} dx$. This is in
sharp contrast with the Airy determinantal process description of
top eigenvalues for the case of light tailed entries \cite{sosh} but
in perfect agreement with our result about the tail of
$\mu_{\alpha}$ given in Theorem \ref{main3}.

\section{Truncating the entries}\label{cutoff}

Since the entries of our random matrices have very few moments, it
will be of importance later to truncate them. We introduce the
appropriate truncated matrices in this section and show how their
spectral measure approximate the spectral measure of the original
matrices.

Let us consider $X_N^B$ (resp. $X_N^\kappa$) the Wigner matrix with
entries $x_{ij}1_{|x_{ij}|\le Ba_N}$ for $B>0$, respectively
$x_{ij}1_{|x_{ij}|\le N^{\kappa}a_N}$ for $\kappa>0$. Also define
$$A_N=a_N^{-1} X_N,
\quad A^B_N=a_N^{-1} X_N^B, \quad A^\kappa_N=a_N^{-1}X_N^\kappa$$
Let us remark here that the threshold $a_N$ is precisely the scale
of the largest entry in a row (or a column) of the random matrix
$X_N$, while the scale of the largest entry (or of the largest
eigenvalue) of the whole matrix is $a_N^2$ i.e roughly
$N^{\frac{2}{\alpha}}$.

We want to state that the spectral measures of the matrices $A_N,
  A^B_N$ and $A^\kappa_N$ are very close in a well chosen distance,
compatible with the weak topology. The standard Dudley distance $d$
is defined on $\Pa(\R)$  by
$$d(\mu,\nu)=\sup_{ ||f||_\La\le 1}\left| \int f d\nu- \int f d\mu\right|$$
where the supremum is taken over all Lipschitz functions $f$ on $\R$ such
that $\|f\|_\La \leq 1$, where the norm $\|f\|_\La$ is defined by
$$\|f\|_\La:=\sup_{x\neq y}\frac{|f(x)-f(y)|}{|x-y|}+
\sup_x|f(x)|.$$ We will use the following variant $d_1 $ of the
Dudley distance.

$$d_1(\mu,\nu)=\sup_{ ||f||_\La\le 1, f\uparrow}
\left| \int f d\nu- \int f d\mu\right|
$$
where the supremum is taken over non-decreasing Lipschitz functions
such that $\|f\|_\La \leq 1$ . The Dudley distance $d$ is well known
to be a metric compatible with the weak topology and the following
Lemma shows that so is the variant $d_1$.

\begin{lem}\label{d1weak}
$d_1$ is compatible with the weak topology on $\Pa(\R)$, i.e if
$\mu$ is a positive measure on $\R$ such that there exists
$\mu^n\in\Pa(\R)$ so that
$$\lim_{n\ra\infty} d_1(\mu^n,\mu)=0,$$
then $\mu^n$ converges weakly to $\mu$ and $\mu\in\Pa(\R)$.
Reciprocally, if $\mu_n$ converges to $\mu$ weakly, $d_1(\mu_n,\mu)$
goes to zero. If a sequence $\mu_n\in\Pa(\R)$ is Cauchy for $d_1$,
it converges weakly.
\end{lem}
\begin{proof} A compactly supported Lipschitz function $f$  can be
written as
$$f(x)= f(0)+ \int_0^x g(y)dy$$
where $g$ is a borelian function bounded by the Lipschitz norm of
$f$. Writing
$$f(x)-f(0)=\int_0^ x 1_{g(y)\ge 0} g(y) dy
-\int_0^x |g(y)|1_{g(y)<0} dy$$ we see that f can be written as the
difference of two non-decreasing Lipschitz functions. Hence, if
$d_1(\mu^n,\mu)$ goes to zero as $n$ goes to infinity, $\int
fd\mu_n$ converges to $\int f d\mu$ for all Lipschitz compactly
supported functions. Hence, $\mu_n$ converges to $\mu$ for the vague
topology. On the other hand, if $\mu^n$ converges to $\mu$ for
$d_1$, we must have, taking $f=1$,
$$\mu(1)=\lim_{n\ra\infty}\mu^n(1)=1$$
which is enough to guarantee also the weak convergence.
Indeed, if we now take $f\in\Ca_b(\R)$, and $g$ compactly supported
with values in $[0,1]$,
$$|\mu_n(f)-\mu(f)|\le \|f\|_\infty\left(\mu(1)+\mu_n(1)-\mu(g)
-\mu_n(g)\right)+|\mu_n(fg)-\mu(fg)|$$ Letting first $n$ going to
infinity and then taking $g$ approximating the unit, we obtain the
result. The second statement is clear since $d_1\le d$ with $d$ the
standard Dudley distance (obtained by taking the supremum over all
Lipschitz functions with norm bounded by one) and the result is well
known to hold for $d$. Finally, if a sequence $\mu_n$ is Cauchy for
$d_1$, it converges for the vague topology (as it is tight for the
vague topology, and the property of being Cauchy uniquely prescribes
the limit) and then for the weak topology by the mass property.

\end{proof}

We next show that truncation does not affect much
the spectral measures in the $d_1$ distance.

\begin{theo}\label{corconv}
\begin{enumerate}
\item For every $\e>0$ there exists  $B(\e)<\infty$ and $\d(\e,B)>0$ when $B>B(\e)$
such that, for N large enough
$$\P\left(d_1(\mun_{A_N},\mun_{A_N^B})>\e\right)
\le e^{-\d(\e,B) N}.$$
\item
For $\kappa>0$, and $a\in]1-\alpha \kappa,1[$, there exists a finite
constant $C(\alpha,\kappa,a)$ such that for all $N\in\N$,
$$
\P\left(d_1(\mun_{A_N},\mun_{A_N^\kappa})>N^{a-1}\right) \le
e^{-CN^a\log N }. $$
\end{enumerate}
\end{theo}

\begin{rmk}
This result depends crucially on the proper choice of the truncation
level. Had we truncated the entries at a lower level, say
$N^{\kappa}a_N$ with $\kappa < 0$, then the limit law would be the
semi-circle. Thus the effect of the heavy tails would have been
completely canceled by the truncation.
\end{rmk}

\begin{proof}

Let X and Y be two  $N\times N$ Hermitian matrices, and $\mun_X$ and
$\mun_Y$ be their spectral measures. Then  Lidskii's theorem implies
(see e.g \cite{GZe} p. 500) that, if $d$ is the rank of $X-Y$,then
\begin{equation}
\label{Lidskii}
d_1(\mun_{X},\mun_{Y})\le \frac{2 d}{N}
\end{equation}

Consequently, the following Lemma implies Theorem (\ref{corconv}).
\end{proof}

\begin{lem}\label{tronc}
\begin{enumerate}
\item
For every $\e>0$, there exists  $B(\e)>0$ and $\d(\e,B)>0$ when
$B>B(\e)$ such that
$$\P(\mbox{rank}(X_N-X_N^B)\ge \e N)\le e^{-\d(\e,B) N}$$

\item
For $\kappa>0$, and $a\in]1-\alpha \kappa,1[$ there exists a finite
constant $C(\alpha,\kappa,a)$ such that for all $N\in\N$,
\begin{equation}
\P(\mbox{rank}(X_N-X_N^\kappa)\ge N^a)\le e^{-C N^a\log N}
\end{equation} .
\end{enumerate}
\end{lem}

\begin{proof}(of Lemma \ref{tronc}.)

Let $M_i^-=1$ (resp. $M_i^+=1$) if there exists a $j \le i$ (resp.
$j >i$ such that $|x_{ij}|> Ba_N$, and $M_i^-=0$ (resp. $M_i^+=0$)
otherwise. Define
$$ M^-=\sum_{i=1}^N M_i^-\,\mbox{ and }\, M^+=\sum_{i=1}^N M_i^+.$$
Now let M be the number of non zero rows of the matrix $X_N-X_N^B$,
obviously
\begin{equation}
\label{estofrank}
\mbox{rank}(X_N-X_N^B) \le M \le M^- + M^+,
\end{equation}
so that
$$
\P(\mbox{rank}(X_N-X_N^B)\ge \e N)\le \P(M^- \ge \frac{\e N}{2}) +
\P(M^+ \ge \frac{\e N}{2})\le 2\P(M^- \ge \frac{\e N}{2}).
$$ where we 
observed that $M^+$ is stochastically dominated by 
$M^-$ (which contains the diagonal terms).
But if we denote by $p_i=\P(M_i^-=1)$, we have
$$
p_i=\P(\exists j \le i, |x_{ij}|>
Ba_N)=1-(1-\frac{L(Ba_N)}{(Ba_N)^{\alpha}})^i\le 1-
(1-\frac{c}{NB^{\alpha}})^i
$$
where the later inequality holds for $c>1$ when $N$ is large enough
since
\begin{equation}
\label{estdebase}
\lim_{N\ra\infty}\frac{NL(Ba_N)}{a_N^{\alpha}}=1.
\end{equation}
As a consequence  we can estimate the sum
\begin{equation}
\label{p_i}
\sum_{i=1}^N p_i \le N-
\frac{1-(1-\frac{c}{NB^{\alpha}})^{N+1}}{1-(1-\frac{c}{NB^{\alpha}})}\sim
NC(B)
\end{equation}
where we denoted
$A_N\sim B_N$ if $A_N/B_N$ goes to one as $N$ goes to infinity and
\begin{equation}
C(B)=1-\frac{B^{\alpha}}{c}(1-e^{-\frac{c}{B^{\alpha}}}).
\end{equation}
For any $\lambda>0$, the independence of the $M^-_i$'s gives
$$
\E(\exp\lambda M^-)= \prod_{i=1}^N (1+p_i(e^{\lambda}-1)) \le
\exp[(e^{\lambda}-1)(\sum_{i=1}^N p_i)]
$$
So that we get the exponential upper bound, for N large enough
$$
\P(M^- \ge \frac{\e N}{2})\le e^{-\lambda\frac{\e N}{2}}
\E(\exp\lambda M^-)\le \exp[-N \phi_-(\lambda, \e, B)],
$$
with
$$
\phi_{-}(\lambda, \e, B)= \frac{\lambda\e}{2}
-(e^{\lambda}-1)C(B).
$$
Obviously, since $\lim_{B\ra\infty}C(B)=0$,
for any $\e>0$,  there exists a $B(\e)>0$ (of order
$\e^{-\frac{1}{\alpha}}$) such that when $B>B(\e)$,
$$\delta_{-}(\e, B):=\sup_{\lambda>0}\phi_{-}(\lambda, \e, B)>0$$
and
$$
\P(M^- \ge \frac{\e N}{2})\le  \exp[-N \delta_{-}(\e, B)].
$$
 Using the crude rank
estimate (\ref{estofrank})  proves the first claim of Lemma
(\ref{tronc}).

In order to prove the second claim of Lemma (\ref{tronc}), we simply
replace $B$ by $B(N)=N^{\kappa}$ and $\e$ by $\e(N)=N^{a-1}$ in the
proof above.We get then that
$$
\delta_{-}(\e(N), B(N))\sim \frac{1}{2}(a-1+\alpha\kappa)(N^{a-1} \log
N)
$$
and similarly for $\delta_{+}(\e(N), B(N)$, which proves our second
claim.

\end{proof}

\begin{rmk}
\label{centering}We now let
$A_N^\kappa=a_N^{-1}X_N^\kappa$.
We
note that centering the entries of the matrix $A_N^{\kappa}$ defines
a perturbation of rank one.Hence,  Lidskii's theorem
(see (\ref{Lidskii})) shows that
$$d_1(\mun_{A_N^\kappa},\mun_{A_N^\kappa-\E[A_N^\kappa]})\le
\frac{2}{N}. $$

Thus we may assume that $A_N^\kappa$ is centered
without changing its limiting spectral distribution.
\end{rmk}

\section{Tightness}\label{sectiontight}

We prove in this section that the mean of the spectral measures of
the random matrices $A_N$ and of their truncated versions $A^B_N$ or
$A_N^\kappa$ are tight.

\begin{lem}
\label{tight}
\begin{enumerate}
\item
The sequence $(\E[\mun_{A_N}];N\in\N)$ is tight for the weak
topology on $\Pa(\R)$.
\item
For every $B>0$, and $\kappa>0$, the sequences
$(\E[\mun_{A^B_N}]; N\in\N)$ and
$(\E[\mun_{A_N^\kappa}]); N\in\N)$ are tight for
the weak topology on $\Pa(\R)$.
\end{enumerate}
\end{lem}

\begin{proof}
We will use the following classical result about
truncated moments (Theorem VIII.9.2 of \cite{Feller}): For any $
\zeta > \alpha$
\begin{equation}
\label{truncatedmoments}
\lim_{t\ra\infty}\frac{\E[|x_{ij}|^{\zeta}1_{|x_{ij}|<t}]}{t^{\zeta-\alpha}L(t)}=
\frac{\alpha}{\zeta-\alpha}.
\end{equation}
Therefore, using (\ref{estdebase}), we have
\begin{equation}
\label{truncatedmoments2} \E[|x_{ij}|^{\zeta}1_{|x_{ij}|<Ba_N}] \sim
\frac{\alpha}{\zeta-\alpha}B^{\zeta-\alpha}\frac{a_N^{\zeta}}{N}
\end{equation}
or  equivalently:

\begin{equation}
\label{truncatedmoments3} \E[|A^B_N(ij)|^{\zeta}] \sim
\frac{\alpha}{\zeta-\alpha}B^{\zeta-\alpha}\frac{1}{N}.
\end{equation}
The version for the truncated matrix $A^{\kappa}_N$ will also be
useful:
\begin{equation}
\label{truncatedmoments4} \E[|A^{\kappa}_N(ij)|^{\zeta}] \sim
\frac{\alpha}{\zeta-\alpha}N^{\kappa(\zeta-\alpha)-1}.
\end{equation}
Using these estimates with $\zeta=2$, one sees that

\begin{equation}\label{contcov2}
\sup_{N\in\N}\E[\frac{1}{N}\tr( (A_N^B)^2)]\sim
\frac{\alpha}{2-\alpha}B^{2-\alpha}
\end{equation}
and that

\begin{equation}\label{contcov3}
\sup_{N\in\N}\E[\frac{1}{N}\tr( (A_N^{\kappa})^2)]\sim
\frac{\alpha}{2-\alpha}N^{\kappa(2-\alpha)}.
\end{equation}
\eqref{contcov2}
shows that  $\E[\mun_{A^B_N}]$ belongs to the compact set
$K_C:=\{\mu\in\Pa(\R); \mu(x^2)\le C\}$ for any
$C>\frac{\alpha}{2-\alpha}B^{2-\alpha}$ and N large enough.
Hence,  the sequence $(\E[\mun_{A^B_N}]); N\in\N)$ is tight, and
thus  any subsequence of $\E[\mun_{A^B_N}]$ has converging
subsequences. We denote by $\mu_B$ a limit point, i.e the limit of a
converging subsequence. By a diagonal procedure, we can insure that
this subsequence is the same for all $B\in\N$, and in particular,
since $d_1$ is compatible with the weak topology, we can find an
increasing function $\phi$ so that for any $\d>0$, $B_0<\infty$,
there exists $N_0<\infty$ so that for $N\ge N_0$, and all $B\le
B_0$,
$$
d_1(\E[\hat \mu_{A_{\phi(N)}^{B}}],\mu_B)\le
\delta.$$
By Lemma \ref{tronc}, and Lidskii's estimate (\ref{Lidskii}), we
have for all $\e>0$,
\begin{equation}\label{muBCauchy0}
d_1(\E[\hat \mu_{A_{\phi(N)}}],\E[\hat\mu_{A^B_{\phi(N)}}])\le \E[
d_1(\hat\mu_{A_{\phi(N)}}, \hat\mu_{A^B_{\phi(N)}})]\le
2\e +e^{-\d(\e,B) {\phi(N)}}
\end{equation}
with $\d(\e,B)>0$ if
$B>B(\e)$.

These two inequalities imply that $(\mu_B, B\in \N)$ is a Cauchy
sequence for the modified Dudley metric $d_1$ and thus converges
when $B$ tends to $\infty$.
Indeed, if we choose $\e,\e',\d>0$ and an integer
number $B_0> B(\e)\vee B(\e')$, we
find that for $B,B'\in [B(\e)\vee B(\e'), B_0]$ and $N>N_0$
\begin{equation}\label{ineql}
d_1(\E[\hat\mu_{A_{\phi(N)} }], \mu_B)\le \delta+2\e +e^{-\d(\e,B)
\phi(N)} \mbox{ and } d_1(\E[\hat\mu_{A_{\phi(N)} }], \mu_{B'})\le
\delta+2\e' +e^{-\d(\e',B') \phi(N)}
\end{equation}
and therefore
\begin{equation}\label{muBCauchy}
d_1(\mu_B,\mu_{B'})\le 2\d +2\e+2\e' +e^{-\d(\e,B)
\phi(N)} +e^{-\d(\e',B')\phi(N)}.
\end{equation}
Letting $N$ going to infinity,
and then $\d$ to zero and $B_0$ to
infinity  we finally deduce that
$$d_1(\mu_B,\mu_{B'})\le 2\e+2\e'$$
provided that $B$ and $B'$ are greater than $B(\e)\vee B(\e')$.
Hence,  $\mu_B$ is a Cauchy sequence for $d_1$ and thus converges
weakly by Lemma \ref{d1weak} as $B$ goes to infinity. 
 As a consequence   of \eqref{ineql} we also find that
$\E[\mun_{A_{\phi(N)}}]$ converges to this  limit as $N$ goes to
infinity. The same holds for the truncated versions
$\E[\hat\mu^{\phi(N)}_{A^\kappa_{\phi(N)}}]$. Thus, we have proved
that $(\E[\mun_{A_{N}}], \E[\mun_{A^\kappa_N}])_{N\in\N}$ are tight.
\end{proof} 

This lemma (\ref{tight}) can be strengthened into a partial
almost-sure tightness result. Consider an increasing function
$\phi:\N\ra\N$ such that $\sum_{N\ge 0}\frac{1}{\phi(N)}<\infty$,
then
\begin{lem}
\label{astight} The sequences
$(\hat\mu_{A^B_{\phi(N)}})_{N\in\N},(\hat\mu_{A_{\phi(N)}})_{N\in\N}
,(\hat\mu_{A^\kappa_{\phi(N)}}
)_{N\in\N}$ are  almost surely tight.

\end{lem}
\begin{proof}
We note that the truncated moments bound given in
(\ref{truncatedmoments3}) can be strengthened into a bound in
probability as follows. Let $M>0$ and
$C>\frac{\alpha}{2-\alpha}B^{2-\alpha}$, Chebychev's inequality
reads
\begin{eqnarray}
\P\left(\frac{1}{N}\tr((A_N^B)^2)\ge M+C\right) &\le&\frac{1}{M^2}
\E\left\lbk \left( \frac{1}{N}\tr((A_N^B)^2)-
\E[\frac{1}{N}\tr((A_N^B)^2)]\right)^2\right\rbk\nonumber\\
&=&\frac{1}{M^2} \E\left\lbk \left( \frac{1}{N^2} \sum_{i,j=1}^N
(A_N^B(i,j)^2 -\E[ A_N^B(i,j)^2]
)\right)^2\right\rbk\nonumber\\
&\le& \frac{4}{M^2 N^2}\sum_{i\le j} \E\left\lbk \left( A_N^B(i,j)^2
-\E[ A_N^B(i,j)^2]
\right)^2\right\rbk\\
&\le& \frac{2}{M^2}\max_{i\le j}
\E\left\lbk  A_N^B(ij)^4 \right\rbk\nonumber\\
&\sim & \frac{2\alpha B^{4-\alpha}}{4-\alpha}\frac{1}{M^2N}\nonumber\\
\nonumber
\end{eqnarray}
where we used the independence of the entries at the third step and
the truncated moments estimate (\ref{truncatedmoments3}) for
$\zeta=4$ at the last step. Then Borel Cantelli's  lemma implies
that for any $C>\frac{\alpha}{2-\alpha}B^{2-\alpha}$
$$\limsup_{N\ra\infty}
\frac{1}{\phi(N)}\tr ( (A_{\phi(N)}^B)^2)\le C \quad a.s$$ which
insures the almost sure tightness of
$(\hat\mu^{\phi(N)}_{A^B_{\phi(N)}})_{N\in\N}$. From this point, all the
above arguments apply to show the almost sure tightness of
$(\hat\mu^{\phi(N)}_{A_{\phi(N)}})_{N\in\N}$ and
$(\hat\mu^{\phi(N)}_{A^{\kappa}_{\phi(N)}})_{N\in\N}$.
\end{proof}

\section{Induction over the dimension of the matrices}\label{induction}

We borrow the following idea from \cite{CB}:  in order to prove the
vague convergence of $(\E[ \mun_{A_N}])_{N\in\N} $ we study the
asymptotic behavior, for $z$ a complex number, of the probability
measure $L_N^z$ on $\C$ given, for $f\in\Ca_b(\C)$, by

$$L_N^z(f)= \E\left[
\frac{1}{N} \sum_{k=1}^N f\left(
((z-A_N)^{-1})_{kk}\right)\right].$$

Here and below, $z$ denotes in short $z$ times the identity 
in the set of matrices under consideration.
$L_N^z$ is thus the empirical measure of the diagonal entries of the
resolvent of $A_N$. In contrast to \cite{CB}, we will only consider
these measures when $z\in \C\backslash\R$, where everything is well
defined since $z-A_N$ is invertible.

Note that for $z\in\C^+=\{z\in\C: \Im z>0\}$, and for
$k\in\{1,\cdots, N\}$, the diagonal term $((z-A_N)^{-1})_{kk}$
belongs to the set $ D:=\C^-\cap \{x\in\C: |x|\le |\Im(z)|^{-1}\}$.
$L_N^z$ is thus a probability measure on the compact subset $D$ of
$\C$.

If we choose the function $f(x)=x$ then
$$L_N^z(f)=\E[\frac{1}{N}\tr( (z-A)^{-1})]$$
is the Stieltjes transform of $\E[  \mun_{A_N}]$.

Thus, the weak convergence of $L_N^z$ for all $z\in\C^+$ (or even
for all $z$ in a set with accumulation points) would be enough to
prove the vague convergence of $\E[ \mun_{A_N}]$. Indeed the latter
is a consequence of the convergence of its Stieltjes transform,
which, as an analytic function on $\C^+$, is uniquely determined by
its values on a set with accumulation points.

In the following, given a $z\in \C^+$, we will prove an equation on
the limit points of $L_N^z$ (more precisely of its analogue where
$A_N$ is replaced by its truncation $A_N^\kappa$
for some well chosen $\kappa>0$). Our main tool will be a recursion
on the dimension N,
and the Schur complement formula. We first investigate how these
measures depend on the dimension.

We let $\bar A_{N+1}$ be the $(N+1)\ts (N+1)$ matrix obtained by
adding to $A_N$ a first row and a first column $A_N(0,k)=A_N(k,0)=
a_N^{-1} x_{0k}$. Hence, $\bar A_{N+1}$ has the same
law as $\frac{a_{N+1}}{a_N} A_{N+1}$.

We then let $\hat A_N$ be the $(N+1)\ts (N+1)$ matrix obtained by
adding as first row and column the zero vector.

We also define for $z\in\C\backslash \R$,
$$\bar G_{N+1}(z):=(z- \bar A_{N+1})^{-1}\quad
G_N(z)=(z-A_N)^{-1}\quad \hat G_N(z)=(z-\hat A_N)^{-1}$$

We finally denote by $.^\kappa$ all quantities where $A_N$ has been
replaced by its truncated version $A_N^\kappa$. Thus for
$z\in\C\backslash \R$ we define $$L_N^{z,\kappa}=\frac{1}{N}
\sum_{k=1}^N \delta_{G_N^\kappa(z)_{kk}},\quad \hat
L_{N}^{z,\kappa}=\frac{1}{N+1} \sum_{k=0}^N \delta_{\hat
G_N^\kappa(z)_{kk}},\quad  \bar L_{N+1}^{z,\kappa}= \frac{1}{N}
\sum_{k=1}^N \delta_{\bar G_{N+1}(z)_{kk}}$$.

\begin{lem}\label{approx}

\begin{enumerate}
\item $\hat G_N^\kappa(z)_{kk}$ is equal to $G_N^\kappa(z)_{kk}$ for
$k\ge 1$ and to $z^{-1}$ for $k=0$.
\item
$$\lim_{N\ra\infty}\frac{1}{N}\sum_{k=1}^N \E[ |\bar G_{N+1}^\kappa(z)_{kk}
-\hat G_N^\kappa(z)_{kk}|]=0.$$

\item For $\kappa\in ]0,\frac{1}{2-\alpha}[$ and $
0<\eta<\frac{1}{2}(1-\kappa(2-\alpha))$,
$$\lim_{N\ra\infty}
\P( d(L_N^{z,\kappa},\bar L_{N+1}^{z,\kappa})>N^{-\eta})=0
$$
Here,as above, $d$ is the Dudley distance on $\Pa(\C)$.
\end{enumerate}
\end{lem}
{\bf Proof.} We note that
\begin{equation}\label{pok}
(z- \hat A_N^\kappa)=\left(
\begin{array}{cc}
z &0\cr
0& z-A_N^\kappa\cr
\end{array}\right)\Rightarrow \hat G_N^\kappa(z)=\left(
\begin{array}{cc}
z^{-1} &0\cr
0& (z-A_N^\kappa)^{-1}\cr
\end{array}\right)
\end{equation}
which immediately yields the first point. For the second, let us
write
\begin{eqnarray*}
\bar G_{N+1}^\kappa(z)_{kk}
-\hat G_N^\kappa(z)_{kk}&=& \left(\bar G_{N+1}^\kappa(z)
(\bar A_{N+1}^\kappa- \hat A_N^\kappa) \hat G_N^\kappa(z)\right)_{kk}\\
&=& \sum_{l=0}^N\bar G_{N+1}^\kappa(z)_{kl}
(\bar A_{N+1}^\kappa- \hat A_N^\kappa)_{l0} \hat G_N^\kappa(z)_{0k}\\
&&+ \sum_{l=0}^N\bar G_{N+1}^\kappa(z)_{k0}
(\bar A_{N+1}^\kappa- \hat A_N^\kappa)_{0l} \hat G_N^\kappa(z)_{lk}\\
&=&\bar G^\kappa_{N+1}(z)_{k0}\sum_{l=0}^N
A^\kappa_N({0l}) \hat G^\kappa_N(z)_{lk}\\
\end{eqnarray*}
where we noticed above that $\hat G_N^\kappa(z)_{0k}$
is null for $k\neq 0$ by \eqref{pok}.
Therefore, we find that
\begin{eqnarray*}
\E[|\bar G_{N+1}^\kappa(z)_{kk}
-\hat G_N^\kappa(z)_{kk}|]^2&\le&\E[|\bar
G_{N+1}^\kappa(z)_{k0}|^2]\E[|\sum_{l=0}^N
A^\kappa_N({0l}) \hat G_N^\kappa(z)_{lk}|^2]\\
\end{eqnarray*}
by Cauchy-Schwartz's  inequality. We recall that we have seen in
remark \ref{centering} that we can assume that the entries of the
matrix $A_N^{\kappa}$ are centered. Using then the independence of
$A^\kappa_{0l}$ and $\hat G_N(z)$, summing over $k\in\{1,\cdots,
N\}$ and with a further use of Cauchy-Schwartz's inequality, we find
that,
\begin{eqnarray*}
&&\frac{1}{N}\sum_{k=1}^N\E[|\bar G_{N+1}^\kappa(z)_{kk}
-\hat G_N^\kappa(z)_{kk}|]\\
&&\qquad \qquad \le \max_{j}\E[(A^\kappa_N({0j}) )^2
]^{\frac{1}{2}} \left(\frac{1}{N}\sum_{k=1}^N \E[|\bar
G_{N+1}^\kappa(z)_{k0}|^2]\right)^{\frac{1}{2}}
\E\left[\frac{1}{N}\sum_{l,k=1}^N
| \hat G_N^\kappa(z)_{lk}|^2\right]^{\frac{1}{2}}\\
\end{eqnarray*}

We now note that the entries of the resolvent $\hat G_N(z)$ are
uniformly bounded in modulus. Indeed observe that, if $U$ is a basis
of eigenvectors of $\hat A_N^\kappa$,
with associated eigenvalues $(\l_i,1\le i\le N)\in\R^N$,
for any $k,l\in \{0,\cdots,N\}^2$,
\begin{eqnarray}
|\hat G_N^\kappa(z)_{kl}|&=&|\sum_{r} u_{kr}(z-\lambda_r)^{-1} 
u_{rl}|\nonumber\\
&\le&\frac{1}{|\Im(z)|} (\sum_{r} |u_{kr}|^2)^{\frac{1}{2}}
(\sum_{r} |u_{rl}|^2)^{\frac{1}{2}}\le  \frac{1}{|\Im(z)|} \label{boundG}\\
\nonumber
\end{eqnarray}
and the same holds for $\bar G_{N+1}^\kappa(z)$. Moreover, since the
spectral radius of $\hat G_N(z)$ is bounded above by $1/|\Im(z)|$,
we also have
$$\frac{1}{N}\sum_{l,k=0}^{N}
| \hat G_N^\kappa(z)_{lk}|^2= \frac{1}{N}\tr(  \hat G_N^\kappa(z)\hat 
G_N^\kappa(z)^*)
\le \frac{N+1}{N |\Im(z)|^2}.$$
Hence, we deduce
\begin{eqnarray*}
\frac{1}{N}\sum_{k=1}^N\E[|\bar G_{N+1}^\kappa(z)_{kk} -\hat
G_N^\kappa(z)_{kk}|]&\le&\sqrt{\frac{N+1}{N}}
\frac{1}{|\Im(z)|^2} \max_l \E[(A^\kappa_N({0l}) )^2]^{\frac{1}{2}}.\\
\end{eqnarray*}
But we know how to control the truncated moments
$\E[(A^\kappa_N({0l}) )^2]$. Indeed by the estimate
(\ref{truncatedmoments4}) we see that there exists a finite constant $c$
such that for all $N\in\N$,

\begin{equation}
\max_{1\le i\le j\le N} \E[|A_N^{\kappa}(ij)|^2] \le 
c N^{-\e}
\end{equation}
with $\e=1- \kappa(2-\alpha)>0$. The proof of the second point is
complete.

We finally deduce the last result simply by
\begin{eqnarray}
\E[d(\bar L_{N+1}^{z,\kappa}, \hat L_N^{z,\kappa})]&\le& \frac{1}{N+1}
\sum_{k=0}^N \E[|\bar G_{N+1}^\kappa(z)_{kk} -\hat
G_N^\kappa(z)_{kk}|\wedge 1]\nonumber\\
&\le& \frac{1}{N+1} +
\sqrt{\frac{N+1}{N}}\frac{1}{|\Im(z)|^2} N^{-\frac{\e}{2}}\label{turlut}\\
\nonumber
\end{eqnarray}
and
since $G_N(z)$ and $\hat G_N(z)$ differ at most by a rank one
perturbation,
$$ \hat L_N^{z,\kappa}=\frac{N}{N+1} L_{N}^{z,\kappa}+
\frac{1}{N+1}\d_{z^{-1}}$$ implies that

$$d(L_{N}^{z,\kappa}, \hat L_N^{z,\kappa})\le \frac{2}{N+1}.$$
This shows by Chebychev's inequality that for all
$\eta<\frac{\e}{2}$
$$\lim_{N\ra\infty}
\P(d(L_{N}^{z,\kappa},\bar L_{N+1}^{z,\kappa})>N^{-\eta})=0.$$

\xx

\medskip
To derive an equation for $L_{N}^{z,\kappa}$, our tool will be the
Schur complement formula, which we now recall. Let $\bar A_{N+1}$
and $A_N$ be as above.

\begin{lem}\label{formula}
For any $z\in \C$,
$$\left((\bar A_{N+1}-zI)^{-1}\right)_{00}
=\left( A_N({00})-z -\sum_{k,l=1}^N A_N(0k) A_N(l0)
\left((A_N-zI)^{-1}\right)_{kl}\right)^{-1}.
$$
\end{lem}
\begin{proof}
The proof is a direct consequence of Cramer's inversion formula:
$$\left((\bar A_{N+1}-zI)^{-1}\right)_{00}=
\frac{\det(A_N-zI_{N-1})}{\det(\bar A_{N+1}-zI)}.$$ To get a more
explicit formula for this ratio, write
$$\bar A_{N+1}-zI=\left(\begin{array}{cc}
A_N({00})-z& a_0\\
a_0^T& A_N -z\\
\end{array}\right)$$
with $a_0=(A({01}),\cdots, A({0N}))$, and use   the representation
$$\left[\begin{array}{cc}
I&-BD^{-1}\\
0& I
\end{array}
\right]
\cdot
\left[\begin{array}{cc}
A&B\\
C& D
\end{array}
\right]
=
\left[\begin{array}{cc}
A-BD^{-1}C&0\\
C& D
\end{array}
\right]$$
with
$A=A({00})-z$, $B=a_0$,
$C=a_0^T$ and $D= A_N-z$.
Therefore, as $\det(AB)=\det(A)\det(B)$,
we conclude that
\begin{eqnarray*}
\det(\bar A_{N+1} -z I )&=& \det(A_N -z
I)\det\left[A({00})-z-\langle a_0, (A_N-z I)^{-1} a_0\rangle
\right]\,.
\end{eqnarray*}
This proves the lemma.
\end{proof}
\medskip
We now show that, in the Schur complement formula above, the
off-diagonal terms in the sum in the right hand side are negligible.

\begin{lem}\label{offdiag} For any $\d>0$, any
 $z$ with $|\Im(z)|\ge \d$, any $0< \kappa <
\frac{1}{2(2-\alpha)}$, and $R>0$
$$\P( |\sum_{k\neq l} A_N^\kappa(0k) A_N^\kappa(0l)
\left((A_N^\kappa-z)^{-1}\right)_{kl} |> R)\le \frac{2
 }{R^2
N^{2\e-1}\delta^2}.$$ with $\e= 1-\kappa(2-\alpha)> \frac{1}{2}$.
\end{lem}

\begin{proof}
Following Remark \ref{centering},
 we can always assume that the entries of  $A_N^\kappa$ are
centered. By independence of $ A_N^\kappa(0k)$ and $A_N$, we find
that the first moment of the off-diagonal term vanishes:

$$\E[\sum_{1\le k,l\le N\atop k\neq l} A_N^\kappa(0k) A_N^\kappa(l0)
\left((A_N^\kappa-zI)^{-1}\right)_{kl}]=0 $$ and that the second
moment is small:
\begin{eqnarray*}\E[|\sum_{k\neq l} A_N^\kappa(0k) A_N^\kappa(l0)
\left((A_N^\kappa-z)^{-1}\right)_{kl}|^2] &\le& 2(\max_{i,j}
\E[(A_N^\kappa(ij))^2])^2\E[\sum_{k,l}
|\left((A_N^\kappa-z)^{-1}\right)_{kl}|^2]
\\
& \le & 2 N^{-2\e} \E[\tr( (A_N^\kappa-z)^{-1}(A_N^\kappa-\bar
z)^{-1})] \le \frac{2N^{-2\e+1}}{|\Im(z)|^2}.
\\
\end{eqnarray*}
Chebychev's inequality concludes the proof.
\end{proof}

\medskip

We finally derive from
the previous considerations
a first approximation result
for $L^{z,\kappa}_N$. This will be our first step to obtain a closed
equation for the limit points of the spectral measure (such an
equation will be derived in the next section).

\begin{lem}
\label{approxss} For $0< \kappa < \frac{1}{2(2-\alpha)}$, let $\e=
1-\kappa(2-\alpha)> \frac{1}{2}$.Let $z\in \C^+$.
For any  bounded Lipschitz function $f$,
$$\lim_{N\ra\infty}
|\E[L_N^{z,\kappa}(f)]-\E\left[f\left( \left(z- \sum_{k=1}^N
A_N^\kappa(0k)^2 G_N^\kappa(z)_{kk} \right)^{-1}\right)\right]|
=0.$$
\end{lem}

\begin{proof}
It is clear, by Lemma \ref{approx}, that it is sufficient to prove
that, for a constant $c'$,and every Lipschitz function $f$
\begin{equation}
\label{approx2}
|\E[\bar L_{N+1}^{z,	\kappa}(f)]- \E\left[f\left( \left(z- \sum_{k=1}^N
A_N^\kappa(0k)^2 G_N^\kappa(z)_{kk} \right)^{-1}\right)\right]|\le
\frac{c\|f\|_\La}{|\Im (z)|^{\frac 53}N^{\frac{2\e-1}{3}} }
\end{equation}

We have proved above that, for $z\in\C\backslash\R$, there
exists a random variable $\varepsilon_N(z)$, the
sum of the off diagonal terms and $A_N(00)$
$$\P(|\varepsilon_N(z)|\ge R)\le
\frac{8
}{R^2 N^{2\e-1}|\Im(z)|^2}+ \frac{4\alpha}{R^2(2-\alpha) N^{\e} }  $$
 such that
$$\bar G_{N+1}^\kappa(z)_{00}=\left(z- \sum_{k=1}^N A_N^\kappa(0k)^2
G_N^\kappa(z)_{kk} +\varepsilon_N(z)\right)^{-1}$$ In particular we
have for any Lipschitz function $f$,
\begin{equation}\label{eqw}
\E[ f(\bar G_{N+1}^\kappa(z)_{00})]=\E\left[f\left(
\left(z- \sum_{k=1}^N A_N^\kappa(0k)^2 G_N^\kappa(z)_{kk}
+\varepsilon_N(z)\right)^{-1}\right)\right].
\end{equation}

Observe that with $A_N^\kappa=U \mbox{diag}(\l) U^*$,
$$G_N^\kappa(z)_{kk}=\sum_{i=1}^N |u_{ki}|^2 (z-\lambda_i)^{-1}$$
is such that
$$\Im (z) \Im\left( G^N_{00}(z)_{kk}\right)\le 0,\quad
|G_N^\kappa (z)_{kk}|\le |\Im (z)|^{-1}.$$
In particular, we always have
$$\frac{\Im\left(z- \sum_{k=1}^N A_N^\kappa(0k)^2
G_N^\kappa(z)_{kk}\right)}{\Im(z)}\ge 1.$$
Thus, on $|\varepsilon_N(z)|\le |\Im(z)|/2$,
we obtain the control
$$\left|\left(z- \sum_{k=1}^N A_N^\kappa(0k)^2 G_N^\kappa(z)_{kk}
+\varepsilon_N(z)\right)^{-1}-\left(z- \sum_{k=1}^N A_N^\kappa(0k)^2
G_N^\kappa(z)_{kk}
\right)^{-1}\right|\le \frac{2|\varepsilon_N(z)|}{|\Im(z)|^2}.$$
Hence, if $f$ is Lipschitz,

$$\E[ f(\bar G_{N+1}^\kappa(z)_{00})]=\E[
f\left(
\left(z- \sum_{k=1}^N A_N^\kappa(0k)^2 G_N^\kappa(z)_{kk}
\right)^{-1}\right)]+ O(\|f\|_\La )(\E[\frac{|\varepsilon_N(z)|}{|\Im z|^2}
\wedge 1] +  \frac{1}{N^{2\e-1}|\Im(z)|^{4}})$$
where the last error comes from the weight of putting and removing the cutoff
$|\varepsilon_N(z)|\le |\Im(z)|/2$, due to the fact that $\|f\|_\La$ 
also bounds the uniform bound on $f$.
Now, the right hand side does not depend on the choice
of the indices  and so
we have the same estimate for all
$\E[f(\bar G_{N+1}^\kappa(z)_{kk})]$, for $k\in\{0,1,\cdots,N\}$.
Summing the resulting  equalities we find that
$$\E[ \bar L_{N+1}^{z,\kappa}(f)]=\E\left[f\left(
\left(z- \sum_{k=1}^N A_N^\kappa(0k)^2 G_N^\kappa(z)_{kk}
\right)^{-1}\right)\right]+O(\|f\|_\La )(\E[\frac{|\varepsilon_N(z)|}{|\Im z|^2}
\wedge 1] +  \frac{1}{N^{2\e-1}|\Im(z)|^{4}}).$$ 
This proves the estimate
\eqref{approx2} and thus the lemma.

\end{proof}

\section{The limiting equation}\label{limitingeq}

We prove in this section that the limit points of the sequence of
measures $\E[L_N^{z,\kappa}]$ satisfy an implicit equation. This
section will rely heavily on a result about the convergence of sums
of triangular arrays to complex stable laws. We have deferred to
Appendix \ref{sec-conc} the statements and proofs of these
convergence results. We also refer to the same Appendix for
notations and references about complex stable laws.

Hereafter $z\in\C^+$ will be fixed.  We have seen that
$\E[L_N^{z,\gamma}]$ is a compactly supported probability measure on
$\C$ (since its support lies in the open ball with radius
$1/|\Im(z)|$). Therefore,$(\E[L^{z,\gamma}_N])_{N\in\N}$ is tight,
and we denote by $\mu^z$ a limit point. Recall that for $z\in\C^+$,
$\mu^z$ is a probability measure on $\C^-\cap \{|y|\le
1/|\Im(z)|\}$.

In order to state the main result of this section we will need the
following notations. For $t,z\in\C$, we denote by $\langle t,
z\rangle$ the scalar product of $t$ and $z$ seen as vectors in
$\R^2$, i.e $\langle t, z\rangle=\Re(t)\Re(z)+\Im(t)\Im(z)$.
For a probability measure $\mu $
on $\C$, and $t \in\C$, we define the numbers
$\sigma_{\mu,\alpha}(t)$ and $\beta_{\mu,\alpha}(t)$ by:

\begin{equation}
\label{sigma}
\sigma_{\mu,\alpha}(t)=[\frac{1}{C_{\alpha}}\int|\langle
t,z\rangle|^{\alpha}d\mu(z)]^{\frac{1}{\alpha}}
\end{equation}
and
\begin{equation}
\label{beta}
\beta_{\mu,\alpha}(t)=\frac{\int|<t,z>|^{\alpha}sign<t,z>
d\mu(z)}{\int|<t,z>|^{\alpha}d\mu(z)}
\end{equation}
where
\begin{equation}
\label{Calpha} C_{\alpha}^{-1}= \int_0^{\infty} \frac{\sin
x}{x^{\alpha}}dx=
\frac{\Gamma(2-\alpha)\cos(\frac{\pi\alpha}{2})}{1-\alpha}
\end{equation}

\begin{defi}
For a probability measure $\mu $ on $\C$, we define the probability
measure $P^{\mu}$ on $\C$ by its Fourier transform
$$\int e^{i<t,x>} dP^{\mu}(x)=
\exp
[-\sigma_{\mu,\frac{\alpha}{2}}(t)^{\frac{\alpha}{2}}(1-i\beta_{\mu,\frac{\alpha}{2}}(t)
\tan(\frac{\pi \alpha}{4}))]$$
\end{defi}

$P^{\mu}$ is well defined by this Fourier transform, indeed
$P^{\mu}$ is a complex stable distribution. For this description of
$P^{\mu}$ see the appendix \ref{sec-conc}.

We can now state the main result of this section.
\begin{theo}\label{theo-limitpoint}
For $0< \kappa < \frac{1}{2(2-\alpha)}$, the limit points $\mu^z$ of
$\E[L_N^{z,\kappa}]$ satisfy the equation
$$\int f d\mu^z =\int f\left(\frac{1}{z- x}\right) dP^{\mu^z}(x)$$
for every bounded continuous function f.
\end{theo}

\begin{proof}
We consider a subsequence of $(\E[L^{z,\kappa}_N])$ converging to
$\mu^z$, i.e an increasing function $\phi(N)$ such that
$(\E[L^{z,\kappa}_{\phi(N)}])$ converges weakly to $\mu^z$. We
denote by $P_N^z$ the law   of $\sum_{k=1}^N( A_N^\kappa(0k))^2
G^{\kappa}_N(z)_{kk}.$ For $z\in \C^+$, $P_N^z$ is a probability
measure on $\C^-$ since then $G_N(z)_{kk}\in\C^-$ for all $k$. If f
is Lispchitz, Theorem \ref{theo-limitpoint}  is a direct
consequence of the main result of the preceding section, i.e Lemma
\ref{approxss}, and of the next crucial Lemma
\ref{convergencePN}.

\begin{lem}\label{convergencePN}
If $\E[L^{z,\kappa}_{\phi(N)}]$ converges weakly to $\mu^z$ as $N$
goes to infinity,  then $P_{\phi(N)}^z$ converges weakly to
$P^{\mu^z}$ as $N$ goes to infinity.
\end{lem}

It is then easy to see that the statement of Theorem
\ref{theo-limitpoint} extends to any bounded continuous function.
\end{proof}

We now have to prove Lemma (\ref{convergencePN}).
\begin{proof}

We apply first the following concentration result for
$L_N^{z,\kappa}$.
\begin{lem}\label{concentration}

For $\kappa\in (0,\frac{1}{2-\alpha})$, let $\e= 1 -
\kappa(2-\alpha)
>0$. There exists a  finite constant $c$ so that for $z\in
\C\backslash\R$ and any Lispchitz function $f$ on $\C$
$$\P\left(\left| L_N^{z,\kappa}(f)-\E[L_N^{z,\kappa}(f)]\right|\ge \d\right)
\le \frac{c\|f\|_\La^2}{|\Im(z)|^4\d^2}  N^{-\e}$$
\end{lem}

This Lemma shows that since $\E[L^{z,\kappa}_{\phi(N)}]$ converges
weakly to $\mu^z$, then $L^{z,\kappa}_{\phi(N)}$ also converges
almost surely to the non random probability $\mu^z$. From there, one
can apply Theorem \ref{complex stable laws} of Appendix
\ref{sec-conc} or more precisely its extension Theorem \ref{extended
complex stable laws} which has been built to fit exactly our needs
here, when applied to the variables $X_k= A(0,k)^2$
and therefore $\widetilde a_N=a_N^2$. One must simply
notice that the exponent $\alpha$ in Theorem \ref{extended complex
stable laws} must be replaced here by $\frac{\alpha}{2}$. This
concludes the proof of Lemma \ref{convergencePN}.

\noindent
{\it
Proof of Lemma (\ref{concentration}) .} We prove this concentration
lemma using standard martingale decomposition. We assume that $f$
is continuously differentiable, the generalization to any Lipschitz
function being deduced  by density. We put
$$F_N(A^\kappa_N({kl}), k\le l):=
L_N^{z,\kappa}(f)=\frac{1}{N}\sum_{k=1}^N f(G_N(z)_{kk})$$ Let
$n=N(N-1)/2+N$ and index the set $(A^\kappa_N({kl}),
k\le l)$ by $(A^\kappa_i, 1\le i\le N(N-1)/2+N$
 for some lexicographic order. Then, if we let $\Fa_i=\sigma(
A_j^\kappa, 1\le j\le i)$, the independence and identical
distribution of the $A_i^\kappa$'s shows that, if $P_N$ denotes the
law of $A_i^\kappa$ (i.e the properly truncated and normalized
version of $P$),

\begin{eqnarray}
&&\E[ (F_N-\E[F_N])^2]\nonumber\\
&=& \sum_{i=0}^{n-1} \E[ (\E[F_N|\Fa_{i+1}]-\E[F_N|\Fa_i])^2]\nonumber\\
&=&\sum_{i=0}^{n-1} \int \left(\int F_N(x_1,\cdot,x_{i+1} ,
y_{i+2},\cdot, y_n)dP^{\ot n}_N(y) -\int F_N(x_1,\cdot,x_{i} ,
y_{i+1},\cdot, y_n)dP_N^{\ot n}(y)\right)^2 dP_N^{\ot i+1}
(x)\nonumber\\
&\le& \sum_{i=0}^{n-1} \int   ( F_N(x_1,\cdots,x_{i+1} ,\cdots,x_n)
-\int F_N(x_1,\cdots,x_{i}
, y, x_{i+2}\cdots x_n)dP_N(y))^2 dP_N^{\ot n}(x) \nonumber\\
&\le& \sum_{i=0}^{n-1} \|\partial_{x_{i+1}} F_N\|_\infty^2 \int
(x-y)^2 dP_N^{\ot 2}(x,y) \label{truc}
\\
\nonumber
\end{eqnarray}
In our case, for all $k\in \{1,\cdots,N\}$, all
$m,l\in\{1,\cdots,N\}$,
$$\partial_{A_{ml}} f(G_N(z)_{kk})=
f'(G_N(z)_{kk})(G_{N}(z)_{kl}G_{N}(z)_{mk}+G_{N}(z)_{km}G_{N}(z)_{lk})$$
which yields
\begin{eqnarray*}
\partial_{A_{ml}} F_N(A)&=&\frac{1}{N}\sum_{k=1}^N
f'(G_N(z)_{kk})(G_{N}(z)_{kl}G_{N}(z)_{mk}+G_{N}(z)_{km}G_{N}(z)_{lk})\\
&=&\frac{1}{N}\left(
[G_N(z)D(f')G_N(z)]_{ml}+[G_N(z)D(f')G_N(z)]_{lm}\right)
\end{eqnarray*}
with $D(f')$ the diagonal matrix with entries
$(f'(G_N(z)_{kk}))_{1\le k\le N} $. Note that the spectral radius of
$G_N(z)D(f')G_N(z)$ is bounded by $\|f'\|_\infty /|\Im(z)|^2$ and so
since for all $l,m\in\{1,\cdots,N\}^2$
$$\left|[G_N(z)D(f')G_N(z)]_{lm}\right|
\le \|G_N(z)D(f')G_N(z)\|_\infty\le \|f'\|_\infty /|\Im(z)|^2$$ we
conclude that for all $l,m\in\{1,\cdots,N\}^2$,
\begin{eqnarray*}
|\partial_{A_{ml}} F(A)|&\le &\frac{2\|f'\|_\infty
}{N|\Im(z)|^2}.\\
\end{eqnarray*}
Thus, \eqref{truc} shows that
\begin{eqnarray*}
\E[ (F_N-\E[F_N])^2]&\le & \frac{4\|f'\|_\infty ^2}{N^{2}|\Im(z)|^4}
\frac{N^2}{2}\max_{k,l} \E[(A^\kappa_{kl}-\E[A^\kappa_{kl}])^2]
\\
&\le& \frac{2\|f'\|_\infty^2
}{|\Im(z)|^4}N^{-\e}\\
\end{eqnarray*}
where we used the truncated moment
estimate (\ref{truncatedmoments4}). Chebychev's inequality then
provides the announced bound. 

\end{proof}

We now apply Theorem \ref{theo-limitpoint} for a particular choice
of the function f. To this end, we need to define, for any
$\alpha>0$, the usual branch of the power function $x\ra
x^{\alpha}$, which is  the analytic function on $\C\backslash \R^-$
such that $(i)^{\alpha}= e^{i\frac{\pi\alpha}{2}}$. This amounts to
choosing, if $x=r e^{i\theta}$ with $\theta\in ]-\pi,\pi[$,
$$x^\alpha= r^\alpha e^{i \alpha \theta}.$$
This function is analytic on $\C\backslash \R^-$ and extends by
continuity to $x=r e^{i\theta}$ with $\theta$ decreasing to $-\pi$
$$\lim_{\theta\downarrow -\pi} (r e^{i\theta})^{\alpha}
=r^{\alpha}e^{-i\pi\alpha}$$ 
When $x=r e^{i\theta}$ is on the other side of the cut $\R^-$, i.e when
$\theta$ is slightly smaller than $\pi$, the function jumps by a
multiplicative factor $e^{2i\alpha\pi}$. We want to choose in
(\ref{theo-limitpoint}) the analytic function
$f(x)=x^{\frac{\alpha}{2}}$.

\begin{theo}\label{projectionoflimitpoint}
For $0< \kappa < \frac{1}{2(2-\alpha)}$, let $\mu^z$ be a limit
point  of $\E[L_N^{z,\kappa}]$ and define $X_{\mu^z}:=\int
x^{\frac{\alpha}{2}} d\mu^z(x)$. Then
\begin{enumerate}
\item $X_{\mu^z}$ is analytic in $\C^+$ and $|X_z|\le 
\frac{1}{|\Im(z)|^{\frac{\alpha}{2}}}$
\item $X_{\mu^z}$ is a solution of the following equation:
\begin{equation}\label{cocottes}
X_{\mu^z}= i C(\alpha)\int_0^\infty (it)^{\frac{\alpha}{2}-1}
e^{itz} \exp\{-c(\alpha)(it)^{\frac{\alpha}{2}}X_{\mu^z}\}
dt.
\end{equation}
with
$C(\alpha)=\frac{e^{i\frac{\pi\alpha}{2}}}{\Gamma(\frac{\alpha}{2})}$
and $c(\alpha)=\Gamma(1-\frac{\alpha}{2})$.
\end{enumerate}
\end{theo}

\begin{proof}

The first point is obvious.  Indeed, for some increasing function
$\phi$,
$$X_{\mu^z}=\lim_{N\ra\infty} X^{\phi(N)}_z,\quad
X^N_z:= \E[\frac{1}{N}\sum_{k=1}^{N} \left(
(z-A_{N}^\kappa)^{-1}_{kk}\right)^{\frac{\alpha}{2}}].$$ For each
$N$, $X^N_z$ is an analytic function on $\C^+$. Moreover,
$|X^N_z|\le \frac{1}{|\Im(z)|^{\frac{\alpha}{2}}}$ for all $N$. This
entails that any limit point $X_{\mu^z}$ must  also be  analytic in
$\C^+$.

In order to prove the second point and obtain the closed equation
(\ref{cocottes}) we will need the following classical identity:

\begin{lem}\label{repre1}
For all $z\in\C^+$,

$$\left(\frac{1}{z}\right)^{\frac{\alpha}{2}}
=i C(\alpha)\int_0^\infty (it)^{\frac{\alpha}{2}-1} e^{it z} dt$$
with
$C(\alpha)=\frac{e^{i\frac{\pi\alpha}{2}}}{\Gamma(\frac{\alpha}{2})}$

\end{lem}

This Lemma is proven by a simple contour integration, it is also a
consequence  of Lemma \ref{repre2} , proven in the next section (
plug y=0 in the statement of Lemma \ref{repre2}).

By Theorem (\ref{theo-limitpoint}), and since $\mu^z$ and
$P^{\mu^z}$ are supported in $\C^-$, we can write
\begin{eqnarray*}
X_{\mu^z}&=&
\int \left(\frac{1}{z-x}\right)^{\frac{\alpha}{2}}dP^{\mu^z}(x).\\
\end{eqnarray*}

Applying Lemma \ref{repre1} to $z\ra z-x\in\C^+$ for $P^{\mu^z}$
almost all $x$, and integrating over the $x$'s we have, by Fubini's
theorem,
\begin{equation}\label{coco}
X_{\mu^z}= i C(\alpha)\int_0^\infty (it)^{\frac{\alpha}{2}-1}
e^{itz} \int e^{-itx} dP^{\mu^z}(x) dt.
\end{equation}
We now use Theorem \ref{FL transform of complex stable laws} in the
appendix, with $\nu=\mu^z$ here, and replacing $\alpha$ in Theorem
\ref{FL transform of complex stable laws} by $\frac{\alpha}{2}$. We see that:

\begin{equation}
\int e^{-itx} dP^{\mu^z}(x) =\exp\{-c(\alpha) (it)^{\frac{\alpha}{2}
}\int x^{\frac{\alpha}{2}}d\mu^z(x)\}.
\end{equation}
Plugging this equality into \eqref{coco} yields
\begin{equation}
X_{\mu^z}= i C(\alpha)\int_0^\infty (it)^{\frac{\alpha}{2}-1}
e^{itz} \exp\{-c(\alpha) (it)^{\frac{\alpha}{2} }\int
x^{\frac{\alpha}{2}}d\mu^z(x)\}
dt.
\end{equation}
We have obtained the announced closed equation
\begin{equation}
X_{\mu^z}= i C(\alpha)\int_0^\infty (it)^{\frac{\alpha}{2}-1}
e^{itz} \exp\{- c(\alpha)(it)^{\frac{\alpha}{2}}X_{\mu^z}\}
dt.
\end{equation}

\end{proof}

\section{Proofs of Theorem \ref{main} and of Theorem
\ref{main2}}\label{secuniq}

In this section we gather the preceding arguments and prove Theorem
\ref{main} and Theorem \ref{main2}. This proof will be based on the
following uniqueness result for the closed equation
(\ref{cocottes}). We recall the notation
$$g_{\alpha}(y): =
\frac{2}{\alpha} \int_0^\infty e^{-v^{\frac{2}{\alpha}}} e^{-vy}dv =
\int_0^\infty t^{\frac{\alpha}{2}-1} e^{-t} \exp\{-
t^{\frac{\alpha}{2}} y\} dt$$

\begin{theo}\label{uniquenessofX}
\begin{enumerate}
\item
There exists a unique analytic function $X_z$ of $z\in \C^+ $, such
that $|X_z|= O(|Im(z)| ^{-\frac{\alpha}{2}})$ at infinity,
satisfying the equation
\begin{equation}
\label{Xz}
X_{z}= i C(\alpha)\int_0^\infty (it)^{\frac{\alpha}{2}-1}
e^{itz} \exp\{- c(\alpha)(it)^{\frac{\alpha}{2}}X_{z}\}
dt.
\end{equation}
\item
This solution in fact also satisfies: $|X_z|= O(|z|
^{-\frac{\alpha}{2}})$.
\item
If one defines $Y_z:= (-\frac{1}{z})^{\frac{\alpha}{2}} X_z$, then
$Y_z$ is the unique solution of the equation
$$(-z)^{\alpha} Y_z=C(\a)g_{\alpha}(c(\alpha)Y_z).$$
analytic on $\C^+$ and tending to zero at infinity. In fact $|Y_z|=
O(|z|^{-\alpha})$
\end{enumerate}
\end{theo}

\begin{proof}

We already know that there exists such an analytic solution $X_z$.
Indeed we have seen in the preceding section that, if $\mu^z$ is a
limit point, then $X_{\mu^z}$ is such a solution. In order to prove
uniqueness, we will use that:
\begin{lem}\label{repre2}
For all $z\in\C^+$, and any $y \in \C$

$$(-\frac{1}{z})^{\frac{\alpha}{2}}g_{\alpha}(y)
=i \int_0^\infty (it)^{\frac{\alpha}{2}-1} e^{it
z}\exp[-(-z)^{\frac{\alpha}{2}}(it)^{\frac{\alpha}{2}}y] dt$$

\end{lem}

{\bf Proof.} 

We  write $z=re^{i\theta}$ with some $\theta\in ]0,\pi[$. Assume first that $\theta\in ]0,\frac{\pi}{2}[$.
Since
$f(u)=
(u)^{\frac{\alpha}{2}-1}e^{uz}e^{-u^{\frac{\alpha}{2}}[(-z)^{\frac{\alpha}{2}}y]}$
is analytic in $\C\backslash \R^-$, for all $R>0$ finite, its
integral over the contour
$$\Gamma=\{ i t, \e\le t\le R\}\cup \{e^{i\eta} R, \eta \in
[\frac{\pi}{2}, \pi-\theta]\} \cup \{ e^{i\pi -i\theta} t, R\le t\le
\e\} \cup \{e^{i\eta} \e, \eta\in [ \pi-\theta,\frac{\pi}{2}]\}$$
vanishes. Note that $\eta+\theta\in [\frac{\pi}{2}+\theta,\pi]$ so
that $\Re(R e^{i\eta} z)= Rr \cos(\eta+\theta)<0$ for  all $\eta\in
[\frac{\pi}{2}, \pi-\theta]$ and $\theta\in ]0,\frac{\pi}{2}[$.

This shows
that
$$\lim_{R\ra\infty}R
f(e^{i\eta} R)=0\quad\forall \eta\in [\frac{\pi}{2}, \pi-\theta]
\Rightarrow \lim_{R\ra\infty}R\int_{ \eta\in [\frac{\pi}{2},
\pi-\theta]} f(e^{i\eta} R)d\eta=0.$$ Similarly,
$$\limsup_{\e\ra 0}|\int_{ \eta\in [\frac{\pi}{2}, \pi-\theta]}
f(e^{i\eta} \e)d\eta|<\infty \Rightarrow \lim_{\e\ra 0}\e\int_{
\eta\in [\frac{\pi}{2}, \pi-\theta]} f(e^{i\eta} \e)d\eta=0$$ Hence,
letting $R\ra\infty$ and $\e\ra 0$, we find
$$i\int_0^\infty f(it) dt+\int_{+\infty}^0 f( e^{i(\pi-\theta)} t)
e^{i(\pi-\theta)}dt=0.$$ In other words,
\begin{eqnarray*}
i\int_0^\infty (it)^{\frac{\alpha}{2}-1} e^{it
z}e^{-(it)^{\frac{\alpha}{2}}[(-z)^{\frac{\alpha}{2}}y]} dt
&=&-\int_0^\infty (-e^{-i\theta} t)^{\frac{\alpha}{2}-1} e^{-t|z|}
e^{-(-e^{-i\theta}t)^{\frac{\alpha}{2}}[(-z)^{\frac{\alpha}{2}}y]}e^{-i\theta}dt\\
&=& -z^{-1}\int_0^\infty (-z^{-1} t)^{\frac{\alpha}{2}-1} e^{-t}
e^{-(-z^{-1}t)^{\frac{\alpha}{2}}[(-z)^{\frac{\alpha}{2}}y]}dt\\
\end{eqnarray*}
where we finally did the change of variable $t'=|z|t$. 

Noting the obvious facts $(-z^{-1}
t)^{\frac{\alpha}{2}-1}= (-z^{-1})^{\frac{\alpha}{2}-1}
t^{\frac{\alpha}{2}-1}$ and $ (-z^{-1}
)^{\frac{\alpha}{2}-1}(-z)^{\frac{\alpha}{2}-1}=1$,
we thus have proved that
\begin{eqnarray*}
i\int_0^\infty (it)^{\frac{\alpha}{2}-1} e^{it z}\exp\{-
(-z)^{\frac{\alpha}{2}}(it)^{\frac{\alpha}{2}} y\} dt &=&
(-z^{-1})^{\frac{\alpha}{2}} \int_0^\infty
t^{\frac{\alpha}{2}-1} e^{-t} e^{-yt^{\frac{\alpha}{2}}}dt\\
\end{eqnarray*}
which proves the claim when $\theta\in ]0,\frac{\pi}{2}[$. The case $\theta\in [\frac{\pi}{2}, \pi[$ is identical after an immediate modification of the definition of the contour.
It can also be obtained by a trivial analytic extension argument$\square$

\smallskip

By Lemma \ref{repre2} we remark that, if $X_z$ is a solution of the
equation (\ref{Xz}) and if $z=|z|e^{i\theta}$,
\begin{eqnarray}
X_z&=&-e^{-i\theta}
C(\alpha)\int_0^\infty (-e^{-i\theta}t)^{\frac{\alpha}{2}-1}
e^{-t|z|} \exp\{- c(\alpha) (e^{-i\theta}t)^{\frac{\alpha}{2}}X_z\}
dt\label{coc}\\
&=& -\frac{1}{z}  C(\alpha)\int_0^\infty
(-\frac{t}{z})^{\frac{\alpha}{2}-1} e^{-t} \exp\{- c(\alpha)
(-\frac{t}{z})^{\frac{\alpha}{2}}X_z\}
dt\nonumber\\
&=&(-\frac{1}{z})^{\frac{\alpha}{2}}  C(\alpha)\int_0^\infty
t^{\frac{\alpha}{2}-1} e^{-t} \exp\{- c(\alpha) t^{\frac{\alpha}{2}}
(-\frac{1}{z})^{\frac{\alpha}{2}} X_z\}
dt.\label{cocl}\\
\nonumber
\end{eqnarray}
Hence, if $Y_z:= (-\frac{1}{z})^{\frac{\alpha}{2}} X_z$,
we obtain 
\begin{eqnarray}
(-z)^{\alpha} Y_z &=&
C(\alpha)\int_0^\infty
t^{\frac{\alpha}{2}-1} e^{-t} \exp\{- c(\alpha) t^{\frac{\alpha}{2}}
Y_{z}\}
dt.\label{poi}\\
\nonumber
\end{eqnarray}
This equation for $Y_z$ can be written simply as
$$(-z)^{\alpha} Y_z=C(\a)g(c(\alpha)Y_z).$$
We recall that we have assumed that there exists a constant $C_1$
such that $|X_z| \leq C_1 \Im(z)^{-\frac{\alpha}{2}}$.

Now, consider the
function of two complex variables $F(u,y)=ug_{\alpha}(y)-y$.
Obviously $F(0,0)=0$ and $\partial_y F(0,0)=-1$. By the local
implicit function theorem, there exists $\epsilon_1>0$ and
$\epsilon_2>0$, such that for every $u \in \C$ with $|u|<
\epsilon_1$ there exists a unique $y(u) \in \C$ with $|y(u)|<
\epsilon_2$ satisfying the equation $ F(u,y(u))=0$, i.e
$ug_{\alpha}(y(u))=y(u)$. Moreover
\begin{equation}
\label{croissancey} |y(u)|\le C|u|.
\end{equation}
For any $z \in \C^+$, such that $\Im(z)> L$, with $L^{\alpha}>
\frac{1}{C(\alpha)\epsilon_1} \vee \frac{c(\alpha)C_1}{
\epsilon_2}$, then $|X_z| \leq C_1 L^{-\frac{\alpha}{2}} $  so that
$|Y_z| \leq \frac{C_1}{L^{\alpha}} \leq
\frac{\epsilon_2}{c(\alpha}$. Thus for $z \in \C^+$, such that
$\Im(z)> L$  we have that
$$ |\frac{1}{C(\alpha)(-z)^{\alpha}}| \leq \epsilon_1, \quad 
|c(\alpha)Y_z| \leq \epsilon_2$$

Thus the uniqueness in the local implicit function theorem shows
that $Y_z$ is given by $Y_z=\frac{1}{c(\a)} y(
\frac{1}{C(\a)(-z)^{\alpha}})$ and thus that $X_z=
\frac{1}{(-\frac{1}{z})^{\frac{\alpha}{2}}} Y_z$.  Since $X_z$ is
analytic on $ z \in \C^+$ and uniquely determined on the set of $ z
\in \C^+$ such that $\Im(z)> L$ it is uniquely determined. This
proves the claim of uniqueness for $X_z$. Using the bound
(\ref{croissancey}) now proves the improved bound at infinity, i.e
$|X_z|= O( | z|^{-\frac{\alpha}{2}})$. These arguments prove the
second and third statements of the theorem.
\end{proof}

We can now deduce from this last uniqueness result the convergence
of the mean of the normalized trace of the resolvent.

\begin{theo}\label{theo-limitpoint-uniq}
For any  $\kappa\in ]0,\frac{1}{2(2-\alpha)}[$,
any $z\in \C^+ $,
  $\E[\frac{1}{N}\sum_{k=1}^N G^{\kappa}_N(z)_{kk}]$
converges as $N$ goes to infinity to
\begin{equation}
\label{eqG} G_{\alpha}(z):=i\int_0^\infty e^{it z} e^{-c(\alpha)
(it)^{\frac{\alpha}{2}} X_z} dt =-\frac{1}{z}\int_0^\infty e^{-t}
e^{-c(\alpha) t^{\frac{\alpha}{2}} Y_z} dt
\end{equation}

\end{theo}

\begin{proof}
For any $z\in\C^+$ and any limit point $\mu^z$ ,
\begin{eqnarray*}
\int x d\mu^z(x)
&=&\int \frac{1}{z-x} dP^{\mu^z} (x)\\
&=&i\int_0^\infty \int  e^{it(z-x)}dP^{\mu^z} (x) dt\\
&=& i\int_0^\infty e^{itz}
e^{-c(\a)(it)^{\frac{\alpha}{2}} X_z}dt\\
\end{eqnarray*}
The uniqueness of $X_z$ implies that the mean of the resolvent
$\E[N^{-1}\tr(z-A^{\kappa}_N)^{-1}]$ has a unique limit point which
is given by $$G_{\alpha}(z)=i\int_0^\infty e^{itz}
e^{-c(\a)(it)^{\frac{\alpha}{2}} X_z}dt$$ This shows that
$\E[N^{-1}\tr(z-A^{\kappa}_N)^{-1}]$ converges to $G_{\alpha}(z)$.
In order to finish the proof, observe that for $z\in\C^+$, we can
use the same arguments than in the proof of Lemma \ref{repre2} to
see that
\begin{eqnarray}
G_{\alpha}(z)&=&i\int_0^\infty e^{itz}
e^{-c(it)^{\frac{\alpha}{2}} X_z}dt\nonumber\\
&=&-\frac{1}{z}\int_0^\infty e^{-t}
e^{-c(-t z^{-1})^{\frac{\alpha}{2}} X_z}dt\nonumber\\
&=& -\frac{1}{z}\int_0^\infty e^{-t} e^{-c(t )^{\frac{\alpha}{2}}
Y_z}dt\label{eqGlimit}\\
\nonumber
\end{eqnarray}
\end{proof}

This last result enables us to conclude the proof of Theorem
\ref{main} and Theorem \ref{main2}.

\noindent 
{\it Proof of Theorem \ref{main} and Theorem \ref{main2} }

By Lemma \ref{tight}, $\E[\mun_{A_N^\gamma}]$ is tight for the weak
topology. Taking any subsequence, we see that any limit point $\mu$
is such that its Stieltjes transform must be equal to
$G_{\alpha}(z)$ for all $z\in\C^+$.  This prescribes uniquely the
limit point $\mu$ and thus insures the convergence of
$\E[\mun_{A_N^\gamma}]$ towards $\mu\in\Pa(\R)$ so that
$$\int (z-x)^{-1}d\mu(x)=G_{\alpha}(z), z\in\C^+.$$
By Corollary \ref{corconv}, and the fact that
$$d_1(\E[\mun_{A_N^\kappa}], \E[\mun_{A_N}])
\le \E[d_1(\mun_{A_N^\kappa},\mun_{A_N})]$$ we also conclude that
$\E[\mun_{A_N}]$ converges weakly towards $\mu$. By Lemma
\ref{concentration}, for any $z\in\C\backslash\R$,
$L^{z,\kappa}_N(x)=\int (z-x)^{-1} d\mun_{A_N^\kappa}(x)$ converges
in probability towards $G_{\alpha}(z)$. This convergence holds as
well for finite dimensional vectors $(\int (z_i-x)^{-1}
d\mun_{A_N^\kappa}(x), 1\le i\le n)$. Since $\{(z-x)^{-1}, z\in
\C\backslash \R\}$ is dense in the set $\Ca_0(\R)$  of functions on
$\R$ going to zero at infinity, we conclude that $\int f(x)
d\mun_{A_N^\kappa}(x)$ converges in probability towards $\int
f(x)d\mu(x)$ for all $f\in \Ca_0(\R)$. But also
$\mun_{A_N^\kappa}(1)=\mu(1)=1$ and so this vague convergence can be
strengthened in a weak convergence (see the proof of Lemma
\ref{d1weak}). We finally can remove the truncation by $\kappa$ by
using Corollary \ref{corconv}. Again by Lemma \ref{concentration},
$L^{z,\kappa}_N(x)=\int (z-x)^{-1} d\mun_{A_N}(x)$ converges almost
surely along subsequences $\phi(N)$ so that $\sum
\phi(N)^{-\e}<\infty$ by Borel-Cantelli Lemma. As
$\e=\frac{2}{\alpha}-\frac{2-\alpha}{\kappa}$ is as close to one as
wished, for any sequence $\phi(N)$ so that $\sum
\phi(N)^{-\varepsilon}<\infty$ for some $\varepsilon<1$, we can
choose $\kappa$ close enough to one so that
$L^{z,\kappa}_{\phi(N)}(x)$ converges almost surely to $G(z)$. This
entails the almost sure weak  convergence of
$\hat\mu^{\phi(N)}_{A_{\phi(N)}}$ by the same arguments as above.
\hfill\xx

\begin{rmk}
If we could prove  that the equation given in
Theorem \ref{theo-limitpoint} admits a unique solution $\mu^z$, at
least for $z$ in a set large enough, the convergence of
$\E[L^{z,\kappa}_N]$  to this solution would be assured. We cannot
prove this uniqueness result. But as we have seen we do not really
need such a strong uniqueness statement either. We rather have
proved a weaker statement, i.e the uniqueness of $\int xd\mu^z(x)$, 
which already entails the uniqueness of the limit points for
$\E[\int x dL^{z,\kappa}_N(x)]$, i.e the mean Stieltjes transform of
the spectral measure of $A_N^\kappa$. This is sufficient for our
needs but the question of the uniqueness of solutions to the
equation given in Theorem (\ref{theo-limitpoint}) remains
intriguing.
\end{rmk}

\section{Study of the limiting measure. Proof of Theorem
\ref{main3}}\label{studylimit}

In this section, we prove Theorem \ref{main3}. First, the fact that
the limit measure $\mu_{\alpha}$ is symmetric is obvious. It
suffices to consider the case where the entries have symmetric
distributions. To prove the other statements, we need to consider
the limit of $G_{\alpha}(z)$ as $z$ tends to a positive real number
$x$. We first remark that the analytic function $Y_z$ defined on
$\C^+$ is univalent (i.e one-to-one). Indeed this is an obvious
consequence of the equation, valid for $z \in \C^+$:
$$(-z)^{\alpha} Y_z=C(\a)g_{\a}(c(\alpha)Y_z).$$

In order to study the boundary behavior of $G_{\alpha}(z)$, we thus
have to study the boundary behavior of the univalent function $Y_z$.
For $x \in \R$,  the cluster set $Cl(x)$ is defined as the set of
limit points of $Y_z$ when $z$ tends to $x$ (see \cite{collingwood}
or \cite {Pommerenke}). It is easy to see that for any non zero $x
\in \R$ the cluster set $Cl(x)$ is reduced to one point in $\C \cup
\{ \infty \}$. Indeed, assuming w.l.o.g that $x>0$ we have, for any
finite  $v \in Cl(x)$, the equality $ C(\alpha)g_{\alpha}(c(\alpha)
v)= e^{i\pi\alpha}(x)^{\alpha}
v$. If $Cl(x)$ contains two points it is a continuum, i.e a compact
connected set with more that one point (see \cite{collingwood}). By
analytic continuation we would then get the equality $C(\alpha)
g_{\alpha}(c(\alpha)v)= e^{i\pi\alpha}
(x)^{\alpha} v$ for every $v \in \C$ which is false. The only
remaining possibility for $Cl(x)$ is to be reduced to one finite
point or to the point at infinity. We define
$$K'_{\alpha}=\{x \in \R, Cl(x)=\{\infty\} \}$$
We first prove that $K'_{\alpha}$ is bounded. The proof of Theorem
\ref{uniquenessofX}, using the local implicit function theorem at
infinity, shows that $Y_z$ admits and analytic extension to the
set$\{z \in \C, |z|>L\}$ for $L$ large enough, and that this
extension satisfies $|Y_z| =O(|z|^{-\alpha})$. This obviously proves
that, when $|x|>L$, the cluster set $Cl(x)$ is reduced to one finite
point and thus that $K'_{\alpha}$ is bounded.

We consider the complement $U'_{\alpha}$ of $K'_{\alpha}$ . Let $x
\in U'_{\alpha}$ and  $Y_x$ the unique point in the cluster set
$Cl(x)$.  By continuity, for $x>0$, $Y_x$ satisfies the equation
$$e^{-i\pi\a} x^{\alpha} Y_x=C(\a)g(c(\alpha)Y_x).$$

The local implicit function theorem can be applied to this equation
at $(x, Y_x)$, except  for the subset say $F$ of $\R$ where the
derivative vanishes. The exceptional set $F$ must be bounded, since
the derivative does not vanish at infinity, and its points must all
be isolated. Thus $F$ is finite. For any $ x \in U_{\alpha}'
\backslash F$, the implicit function theorem shows that $Y_z$ can be
extended analytically on a complex neighborhood of $x$. Hence
$U_{\alpha}:= U'_{\alpha} \backslash F$ is open and its complement
$K_{\alpha}= K'_{\alpha} \cup F \cup \{0\}$ is closed. $K_{\alpha}$
is also bounded and thus compact.

Finally we use Beurling's Theorem which states that the set
$K'_{\alpha}$ has capacity zero, and thus also the set $K_{\alpha}$
(see \cite{collingwood} or \cite{Pommerenke}).

For any point $x$ in the open set $U_{\alpha}$  the function $Y_z$
admits an analytic extension to a complex neighborhood of $x$, and
thus the Stieltjes transform $G_{\alpha}(z)$ admits a smooth
extension, which proves that $\mu_{\alpha}$ has a smooth density
$\rho_{\alpha}$ on the open set $U_{\alpha}$. Indeed, for $x \in
U_{\alpha}$

$$\lim_{z\ra x}G_{\alpha}(z)=
H\mu_{\alpha}(x)-i\pi \rho_{\alpha}(x)= -\frac{1}{x}\int_0^\infty
e^{-t} e^{-c(\alpha) t^{\frac{\alpha}{2}} Y_x}dt$$

In particular the density of the measure $\mu_{\alpha}$ is given, if
$Y_x= r_x e^{i\phi_x}$, by
\begin{equation}\label{densite}
\rho_{\a}(x)=\frac{1}{\pi x}\int_0^\infty e^{-t} e^{-c(\a)
t^{\frac{\alpha}{2}}[r_x cos(\phi_x)] } \sin[ c(\a)
t^{\frac{\alpha}{2}} r_x sin (\phi_x)] dt.
\end{equation}

Note that we now know that $Y_x$ is well defined and smooth for $x$
large enough. We also have seen that $ Y_x =O(|x|^{-\a})$ and thus
that $ Y_x \sim e^{i\pi\a} C(\a)g_{\a}(0)x^{-\a} $ . Hence, when
$x\ra\infty$, the following asymptotic behavior holds for $
G_{\a}(x)= \lim_{z\ra x}G_{\a}(z)$ :
$$ G_{\a}(x) \sim
\frac{1}{x}\int_0^\infty e^{-t} (1-ct^\frac{\alpha}{2}
Y_x(1+o(1)))dt \approx \frac{1}{x}(1-c \int_0^\infty e^{-t}
t^\frac{\alpha}{2}dt  Y_x(1+o(1)))$$ Identifying the imaginary parts
of both sides we get:
$$\rho_{\a}(x)\sim \pi^{-1} c\Gamma(\alpha)\frac{\Im(Y_x)}{x}.$$
Which proves the last statement of Theorem \ref{main3}.

\section{Cizeau and Bouchaud's  characterization}\label{sectCB}

In \cite{CB}, the authors propose the following argument; they look
at $G_N(z)_{00}$ for $z$ {\it on the real line}. By arguments similar
to those we used (but with no a priori bounds on the $G_N(z)_{kk}$)
they argue that $G_N(z)_{00}$ converges in law as $N$ goes to
infinity. The limit law, that we will denote $P_G$ to follow their
notations (but which is $\mu^z$ in ours) is then given by the 
implicit 
equation (11) in \cite{CB}
$$\int f(y) dP_G(y)=\int f(\frac{1}{z-y}) dP_S(y)
=\int \frac{1}{y^2} f(y) dP_S(z-\frac{1}{y}).$$
$P_S=L_{\alpha}^{C(z),\beta(z)}$ is now a {\it real-valued } stable law
with parameters $C(z)$ and $\beta(z)$
given self-consistently (see (12a) and (12b) in \cite{CB}) by
\begin{eqnarray*}
C(z)&=&\int |y|^{\frac{\alpha}{2}}dP_G(y)=\int |y|^{\frac{\alpha}{2}-2}
dP_S(x-\frac{1}{y})
\\
\beta(z)&=&\int |y|^{\frac{\alpha}{2}}\mbox{sign}(y)dP_G(y)
\\
\end{eqnarray*}
where there was a typographical error  in the definition
of $\beta$ in \cite{CB}. 12b which was already noticed in
\cite{burda}. We in fact have that for
any real $t$,
\begin{eqnarray}
\int e^{-it y} dP_S(y)&=&e^{- C_\alpha^{-1}t^{\frac{\alpha}{2}} (C(z)-
i\tan(\frac{\pi\alpha}{4})\beta(z))}
\nonumber\\
&=&e^{-\Gamma(\alpha-1)(it)^{\frac{\alpha}{2}}
\int (x)^{\frac{\alpha}{2}}d P_G(x)}\label{ftS}\\
\nonumber
\end{eqnarray}
where we used that $K_z:=\int (x)^{\frac{\alpha}{2}}d P_G(x)
=e^{-\frac{i\pi\alpha}{4}}[\cos(\frac{\pi\alpha}{4}) C(z)-i 
\sin(\frac{\pi\alpha}{4}) \beta(z)]$. 
So, we see that the description of the limit law
is very similar to ours, except that $z$ is
supposed to belong to $\R$. 
Let us assume (as seems to be the case 
in \cite{CB}) that 
$C(z)$ and $\beta(z)$
are finite.  Then, also  $K_z$
is finite and 
we see that for non negative real  $z$'s
\begin{eqnarray}
K_z&=& \int (z-y)^{-\frac{\alpha}{2}} d P_S(y)\nonumber\\
&=& -C(\alpha)\int_0^\infty t^{\frac{\alpha}{2}-1}e^{-t z}
e^{-\Gamma(\alpha-1)(it)^{\frac{\alpha}{2}} K_z} dt. \label{cfo}\\
\nonumber
\end{eqnarray}
Hence,  $K_z$ and the
$X_z$ introduced in section
\ref{studylimit}  satisfy formally  the same equation,
except that $X_z$ satisfies it for $z\in\C^+$
and $K_z$ for real $z$'s. Moreover, we have seen 
that $X_z$
can be extended continuously
to $z$ real in  $ (K_\alpha')^c$ and then this extension $X_z$ satisfies
the same equation that $K_z$.
This indicates that we expect $K_z$ and $X_z$ to
be equal, at least on $(K_\alpha')^c$.
In fact, $X_z$ is the unique solution of this equation with an
analytic extension to $\C^+$  and  going to zero at
infinity. In \cite{CB}, under (12a-12b), it is claimed
that the equations defining $C(z), \beta(z)$ have a unique solution,
and so that $K_z$ is also determined uniquely 
by \eqref{cfo}. 
We could not prove
the uniqueness of the solutions to this equation
on the real line. In any case, 
if we beleive either that  $K_z$ extends analytically on
$\C^+$ and goes to zero at infinity
or that the above equation has a unique solution
for $z\in\R$, we must have $X_z=K_z$ at least
for $z\in (K_\alpha')^c$.

The second claim of \cite{CB} is that the density of the limiting
spectral measure $\rho(z)dz=d\mu(z)$ is given, see \cite{CB} (14),
by
$$\rho(z)=\frac{dP_S}{dz}(z).$$
Note that by Fourier inversion, if  $K_z=X_z$,
for $z>0$, since $P_S$ is a probability
measure on $\R$ with Fourier transform given by \eqref{ftS},
\begin{eqnarray*}
\frac{dP_S}{dz}(z)&=&\frac{1}{2\pi} \int_\R e^{-it z}
e^{-\Gamma(\alpha-1)(it)^{\frac{\alpha}{2}} X_z} dt\\
& =&\frac{1}{\pi} \Im\left( \int_0^\infty e^{-itz }
e^{-\Gamma(\alpha-1)(it)^{\frac{\alpha}{2}} X_z} dt\right)
\\
&=&- \frac{1}{\pi}\Im\left(\frac{1}{z}  \int_0^\infty e^{-t }
e^{-\gamma(\alpha-1)(t)^{\frac{\alpha}{2}} Y_z} dt\right)\\
\end{eqnarray*}
and therefore we miraculously recover our result
\eqref{densite}.  Hence, at least for $z\in (K_\alpha')^c$,
the prediction of \cite{CB} coincides 
with our result if we beleive that \eqref{cfo} has a unique solution.

\section{The moment method. Proof of Theorem \ref{main4}}\label{zaka}

We prove here Theorem \ref{main4} using the moment method developed
by I. Zakharevich \cite{zakh}. For any $B>0$, we consider the matrix
$X_N^B$ with truncated entries $x_{ij}^B= x_{ij}1_{|x_{ij}|\le
Ba_N}$ and the normalized matrix $A_N^B=a_N^{-1} X_N^B$. Recall that
work here under the additional hypothesis (\ref{skewness}):
$$ \lim_{u \ra \infty} \frac{\P(x(ij)>u)}
{\P(|x(ij)|>u)}= \theta \in [0,1]$$

We begin by the following estimate on moments of the entries of
$A_N^B$.

\begin{lem}
For any integer $m \ge 1$, the following limit exists
$$C_{m}=\lim_{N \ra \infty}
\frac{\E[A_N^B(ij)^{m}]}{N^{\frac{m}{2}-1}\E[A_N^B(ij)^{2}]^{\frac{m}{2}}}
$$
Moreover, if $m=2k$ is even
$$C_{m}=\frac{2-\alpha}{m-\alpha}
(\frac{2-\alpha}{\alpha}B^{\alpha})^{\frac{m}{2}-1}$$

If $m=2k-1$ is odd
$$C_{m}=(2\theta-1)\frac{2-\alpha}{m-\alpha}
(\frac{2-\alpha}{\alpha}B^{\alpha})^{\frac{m}{2}-1}$$

\end{lem}

\begin{proof}
It is a simple application of the classical result about truncated
moments (Theorem VIII.9.2 of \cite{Feller}) already used in Section
\ref{tight}, \eqref{truncatedmoments2} : For any $ \zeta \ge \alpha$

$$ \E[|x(ij)|^{\zeta}1_{|x(ij)|<Ba_N}] \sim
\frac{\alpha}{\zeta-\alpha}B^{\zeta-\alpha}\frac{a_N^{\zeta}}{N}
$$

The first item of the Lemma is a direct consequence of this estimate
for $\zeta=2$ and $\zeta=2k$. The second is also a consequence of
this estimate, used for $x(ij)^+$ and $x(ij)^-$, and of the
additional skewness hypothesis (\ref{skewness}).

\end{proof}

This lemma enables us to get the main result of this section, i.e
the convergence of the moments of the spectral measure of the matrix
$A_N^B$. We will need some more notations that we take verbatim from
Zakharevich. For any integer $k \ge 1$, we define $V_k$ as the set
of all $(e_1,...,e_l)$such that $\sum_{i=1}^{l}e_i=k$ and $e_1\ge
e_2\le ... \ge e_l >0$. For any $(e_1,...,e_l) \in V_k$ define
$T(e_1,...,e_l)$ as the number of colored rooted trees with $k+1$
vertices and $l+1$ distinct colors, say $(c_1,...c_l)$ satisfying
the following conditions:
\begin{enumerate}
\item
There are exactly $e_i$ nodes of color $c_i$. The root node is the
only node colored $c_0$
\item
If nodes a and b are the same color then the distance from a to the
root is the same as the distance from b to the root
\item
If nodes a and b have the same color then their parents also have
the same color
\end{enumerate}

With these notations we have the following convergence result,
directly implied by Zakharevich's results.

\begin{lem}
\begin{enumerate}
\item
For every integer $k \ge 1$, the following limit exists
\begin{equation}\label{approxmom}
\lim_{N\ra\infty} \E[\int x^k d\mun_{A^B_N}(x)]=:m_k^B
\end{equation}
\item
$m_k^B=0$ if k is odd, and $m_{2k}^B=\sum_{(e_1,...,e_l)\in V_k}
T(e_1,...,e_l)\prod_{i=1}^{l}C_{2e_i} $.
\item
There exists a probability measure $\mu^B_{\alpha}$ uniquely
determined by its moments $m_k^B$. $\mu^B_{\alpha}$ is independent
of the skewness parameter $\theta$.
\item
$\mu^B_{\alpha}$ has unbounded support and is symmetric.
\item
The mean spectral measure $\E[\mun_{A^B_N}]$ converges weakly to
$\mu^B_{\alpha}$.
\end{enumerate}
\end{lem}

\begin{proof}
In order to prove the first and second items, it is enough to use
the preceding Lemma, Corollary 6 and Theorem 2 in \cite{zakh}, plus
the fact that
$$\lim_{N \ra \infty}N \E[ A_N^B(ij)^{2}]=
\frac{\alpha}{2-\alpha}B^{2-\alpha}$$. The third item is a
consequence of the estimate
$$C(m)\le \ C \rho^m$$
with $\rho=(\frac{2-\alpha}{\alpha}B^{\alpha})^{\frac{1}{2}}$ and of
Proposition 10 in \cite{zakh}. The fact that $\mu^B_{\alpha}$ is
independent of the skewness parameter $\theta$ is obvious since its
moments only depend on the $C_m$ for even m's, which are insensitive
to the parameter $\theta$. The fourth item is a consequence of
Proposition 9 and Proposition 12 of \cite{zakh}. The fifth one is a
consequence of Theorem 1 of \cite{zakh}.
\end{proof}

This lemma proves the first part of Theorem \ref{main4}. In order to
prove the second part we simply remark that we have already done so,
since we have seen, in the proof of Lemma \ref{tight}, that
$\mu^B_{\alpha}$ converges and that its limit is the weak limit of
$\E[\mun_{A_N}]$.

\section{Appendix: Convergence to stable distributions for triangular
arrays}\label{sec-conc}

We begin here by recalling the notations for stable distributions,
see for instance \cite{Taqqu}. A real random variable Y has a stable
distribution with exponent $\alpha \in (0,2)$, $\alpha \neq 1$,
scale parameter $\sigma>0$, skewness parameter $\beta \in [-1,1]$,
and shift parameter $\mu \in \R$ (in short
$Stable_{\alpha}(\sigma,\beta,\mu)$ ) iff its characteristic
function is given by:
$$
\E[\exp(itY)]= \exp{[-\sigma^{\alpha}|t|^{\alpha}(1 - i\beta sign(t)
\tan(\frac{\pi\alpha}{2})) + i\mu t]} 
$$
We will consider here only the case where $\alpha < 1$.

A complex random variable Y has an $\alpha$-stable distribution with
spectral representation $(\Gamma, \mu)$ if $\Gamma$ is a finite
measure on the unit circle $S^1$, and $\mu$ is a complex number such
that the characteristic function of Y is given by:
$$
\E[\exp(i \langle t,Y\rangle )]= \exp{[-\int_{S^1}| \langle
t,s\rangle |^{\alpha}(1 - i sign(\langle t,s\rangle)
\tan(\frac{\pi\alpha}{2}))\Gamma(ds) + i \langle \mu, t\rangle]}
$$
We will need the constant
$$
C_{\alpha}^{-1}= \int_0^{\infty} \frac{\sin x}{x^{\alpha}}dx=
\frac{\Gamma(2-\alpha)\cos(\frac{\pi\alpha}{2})}{1-\alpha}
$$
 Throughout this section,
 we consider  a sequence of i.i.d non
negative random variables $(X_k)_{k\geq 1}$ and assume that their
common distribution is in the domain of attraction of an
$\alpha$-stable distribution, with $\alpha\in (0,1)$, i.e that the
tail is regularly varying:
$$
P[ X \geq u] = \frac{L(u)}{u^{\alpha}}
$$
We introduce the normalizing constant $\widetilde a_N$ by:
\begin{equation}
\widetilde a_N= \inf(u, P[ X \geq u]\le \frac{1}{N})
  \label{normalisation22}
  \end{equation}
We consider a triangular array of real or complex numbers
  $(G_{N,k}, 1\le k\le N)$ and give sufficient conditions for the
normalized sum:
$$
  S_N= \frac{1}{\widetilde a_N} \sum_{k=1}^{N} G_{N,k}X_k
$$
  to converge in
  distribution to a (real or complex) stable distribution.
  We will always assume that the triangular array is bounded, i.e
  that
$$
M:= \sup (|G_{N,k}|, N \geq 1, 1\le k\le N) < \infty
$$
We begin with the case where the numbers $G_{N,k}$ are real.

\begin{theo}\label{real stable laws}
Assume that the triangular array of real numbers $(G_{N,k}, N \geq
1, 1\le k\le N)$ is bounded. Furthermore assume that the empirical
measure
$$
  \nu_N= \frac{1}{N} \sum_{k=1}^{N} \delta_{G_{N,k}}
$$
converges weakly to a probability measure $\nu$ on the real line.
Then the distribution of the normalized sum $S_N= \frac{1}{\widetilde a_N}
\sum_{k=1}^{N} G_{N,k}X_k$ converges to a
$Stable_{\alpha}(\sigma,\beta,0)$ distribution, with
$$
\sigma^{\alpha} = \frac{1}{C_{\alpha}}\int|x|^{\alpha}d\nu(x),
$$
$$
\beta=\frac{\int|x|^{\alpha}
sign(x)\nu(dx)}{\int|x|^{\alpha}\nu(dx)}
$$
If $\sigma^{\alpha}=0$, i.e if $\nu = \delta_0$, the above statement
should of course be understood as: $S_N= \frac{1}{\widetilde a_N}
\sum_{k=1}^{N} G_{N,k}X_k$ converges in distribution to zero.
\end{theo}

\begin{proof}[Proof of theorem \ref{real stable laws}]
We begin with the particular case where the numbers $G_{N,k}$ are
positive and bonded below. We assume that there exists an $\delta
 > 0$ such that for any $N \geq 1$ and $1\le k\le N$
\begin{equation}
\delta \leq G_{N,k} \leq M .\label{boundedness}
\end{equation}
In this context we will be able to apply directly classical theorems
to the array of non negative
 independent random variables
$$
U_{N,k}=\frac{1}{\widetilde  a_N} G_{N,k}X_k
$$
For instance, we could apply the theorem in section XVII.7 of
\cite{Feller}. We rather choose to apply Theorem 8, chapter 5 of
\cite{Galambos}. According to this last result, Theorem \ref{real
stable laws} will be proved in this restricted case if we can check
the following three conditions. First the Uniform Asymptotic
Negligibility (UAN) condition, for every $\epsilon>0$
\begin{equation}
\lim_{N\ra\infty} \max_{1\le k\le N} \P(U_{N,k}>\epsilon) = 0.
\label{UAN}
\end{equation}
Second we must check that:
\begin{equation}
\lim_{\epsilon \ra 0}\lim_{N\ra\infty} \sum_{1\le k\le
N}Var[U_{N,k}1_{(U_{N,k}< \epsilon)}]  = 0, \label{variance control}
\end{equation}
and finally we must check that, for $x>0$
\begin{equation}
\lim_{N\ra\infty} \P( \max_{1\le k\le N}U_{N,k} \le x)  = \exp( -
\frac{C_{\alpha}\sigma^{\alpha}}{x^{\alpha}}), \label{right tail
control}
\end{equation} and that
\begin{equation}
\lim_{N \ra\infty} \P( \min_{1\le k\le N}U_{N,k} \le x)  = 1.
\label{left tail control}
\end{equation}
We first note that
$$
\P(U_{N,k}>\epsilon) = \P (X_k > \epsilon \frac{\widetilde a_N}{G_{N,k}}) \le
\frac{L( \frac{\epsilon \widetilde a_N}{G_{N,k}})}{(\frac{\epsilon
\widetilde a_N}{G_{N,k}})^{\alpha}}
$$
which shows that \eqref{UAN} is thus a direct consequence of our
assumption (\ref{boundedness}) and of the following lemma.
\begin{lem}\label{uniform slow variation}
Let L be a slowly varying function and define $\widetilde a_N$ as in
(\ref{normalisation22}):
\begin{equation}
\widetilde a_N= \inf(u, P[ |X| \geq u]\le \frac{1}{N})
\end{equation}
Then , for any $0<a<b$ and any $a<y<b$
\begin{equation}
\frac{L(y\widetilde  a_N)}{(y\widetilde a_N)^{\alpha}}
=\frac{1}{N}\frac{1}{y^{\alpha}}(1+\epsilon(x,N))
\end{equation}
with
\begin{equation}
\lim_{N \ra \infty}\sup_{a<y<b}\epsilon(x,N)= 0
\end{equation}
\end{lem}

\begin{proof}[Proof of Lemma \ref{uniform slow variation}]
Writing
\begin{equation}
\frac{L(y\widetilde a_N)}{(y\widetilde  a_N)^{\alpha}}=
\frac{L(y\widetilde a_N)}{L(\widetilde  a_N)}\frac{NL(\widetilde a_N)}{
\widetilde a_N^{\alpha}}\frac{1}{Ny^{\alpha}}
\end{equation}
this lemma is clearly a direct consequence of the classical fact:
\begin{equation}
\lim_{N \ra \infty}\frac{NL(\widetilde a_N)}{(\widetilde a_N)^{\alpha}}=1
\end{equation}
and of the uniform convergence theorem for slowly varying functions
(\cite{Bingham}, Theorem 1.2.1), which asserts that the convergence
\begin{equation}
\lim_{t \ra \infty}\frac{L(tx)}{L(t)}=1
\end{equation}
is uniform for x's in a compact subset of $(0,\infty)$.
\end{proof}

Next, in order to control the variance $Var[U_{N,k}1_{(U_{N,k}<
\epsilon)}]$ and prove the validity of (\ref{variance control}), we
must use Karamata's theorem, or more directly Theorem VIII.9.2 of
\cite{Feller}which shows that
\begin{equation}
\lim_{t\ra\infty}\frac{t^{\zeta-\alpha}L(t)}{\E[X^{\zeta}1_{X<t}]}=
\frac{\zeta-\alpha}{\alpha}.
\end{equation}
Using this for $\zeta=1,2$,  we see that
\begin{equation}
Var[U_{N,k}1_{(U_{N,k}< \epsilon)}] \backsim
\frac{\alpha}{2-\alpha}\epsilon^2 [\frac{L( \frac{\epsilon
\widetilde a_N}{G_{N,k}})}{(\frac{\epsilon \widetilde a_N}{G_{N,k}})^{\alpha}}].
\end{equation}
Lemma (\ref{uniform slow variation}) then shows that
$Var[U_{N,k}1_{(U_{N,k}< \epsilon)}]$ is of order
$\frac{\epsilon^2}{N}$, and thus that
\begin{equation}
\lim_{\epsilon \ra 0}\lim_{N\ra\infty} \sum_{1\le k\le
N}Var[U_{N,k}1_{(U_{N,k}< \epsilon)}]  = 0
\end{equation}
In order to complete the proof of Theorem \ref{real stable laws}
in the particular case where the numbers $G_{N,k}$ are positive and
bounded below, we now only have to check (\ref{right tail control})
since (\ref{left tail control}) is obvious, the variables $u_{N,k}$
being non negative. For $x>0$
$$
\log \P( \max_{1\le k\le N}U_{N,k} \le x)=\sum_1^N \log[1 - \frac{L(
\frac{x\widetilde a_N}{G_{N,k}})}{(\frac{x\widetilde  a_N}{G_{N,k}})^{\alpha}}].
$$
Using again Lemma(\ref{uniform slow variation}) we see that
$$
\lim_{N\ra\infty}\log \P( \max_{1\le k\le N}U_{N,k} \le x)=-
\frac{C_{\alpha}\sigma^{\alpha}}{x^{\alpha}}
$$
since
$$
\lim_{N\ra\infty}\frac{1}{N}\sum_1^N G_{N,k}^{\alpha} =
\int|x|^{\alpha}\nu(dx)=C_{\alpha}\sigma^{\alpha}.
$$
This checks the condition (\ref{right tail control}) and finishes
the proof in the particular case where the numbers $G_{N,k}$ are
positive and bounded below. Now it is easy to prove theorem
\ref{real stable laws} in full generality. It is enough to split the
sum into the three independent summands
$$
S_N= \frac{1}{\widetilde a_N} \sum_{k=1}^{N} G_{N,k}X_k= \sum_{k=1}^{N}U_{N,k}=
S_N^{+,\epsilon} - S_N^{-,\epsilon} + S_{N,\epsilon}
$$
with
\begin{eqnarray*}
S_N^{+,\epsilon}&=&\sum_{k=1}^{N}U_{N,k}1_{\epsilon<G_{N,k}}\\
S_N^{-,\epsilon}&=& - \sum_{k=1}^{N}U_{N,k}1_{G_{N,k}< -\epsilon}\\
S_{N,\epsilon}&=&\sum_{k=1}^{N}U_{N,k}1_{|G_{N,k}|\le \epsilon}\\
\end{eqnarray*}
We now know that, if $\epsilon$ and $-\epsilon$ are not atoms of
$\nu$, then $S_N^{+,\epsilon}$ (resp $S_N^{-,\epsilon}$) converges
in distribution to a $Stable_{\alpha}(\sigma_{\alpha,
\epsilon}^+,1,0)$ (resp
$Stable_{\alpha}(\sigma_{\alpha,\epsilon}^-,1,0)$) with
\begin{eqnarray*}
C_{\alpha}\sigma_{\alpha,\epsilon}^+ &=&
\int_{\epsilon}^{\infty}|x|^{\alpha}\nu(dx)\\
C_{\alpha}\sigma_{\alpha,\epsilon}^- &=
\int_{-\infty}^{-\epsilon}|x|^{\alpha}\nu(dx)\\
\end{eqnarray*}
So that the sum $S_N^{+,\epsilon}+ S_N^{-,\epsilon}$ converges in
distribution, when N tends to $\infty$ to a
$Stable_{\alpha}(\sigma_{\alpha, \epsilon},\beta_{\alpha,
\epsilon},0)$ with
\begin{eqnarray*}
C_{\alpha}\sigma_{\alpha,\epsilon}&=& \int_{|x|>\epsilon}|x|^{\alpha}\nu(dx)\\
\beta_{\alpha,\epsilon} &=&
\frac{\int_{|x|>\epsilon}|x|^{\alpha}
sign(x)\nu(dx)}{\int_{|x|>\epsilon}|x|^{\alpha}\nu(dx)}\\
\end{eqnarray*}
It is clear that, since $\lim_{\epsilon\ra
0}\sigma_{\alpha,\epsilon}=\sigma_{\alpha}$ and that
$\lim_{\epsilon\ra 0}\beta_{\alpha,\epsilon}=\beta_{\alpha}$, the
distribution
$Stable_{\alpha}(\sigma_{\alpha,\epsilon},\beta_{\alpha,\epsilon},0)$
converge to $Stable_{\alpha}(\sigma_{\alpha},\beta_{\alpha},0)$,
when $\epsilon$ tends to zero. Thus there exists a sequence
$\epsilon_N$ tending to zero such that the sum $S_N^{+,\epsilon_N}+
S_N^{-,\epsilon_N}$ converges in distribution to a
$Stable_{\alpha}(\sigma_{\alpha},\beta_{\alpha},0)$ variable.

But $S_{N,\epsilon_N}$ converges to zero in probability when $N \ra
\infty$. Indeed, for any $x>0$,
$$
\P(|S_{N,\epsilon_N}|> x) \le P( \frac{1}{\widetilde a_N}\sum_{k=1}^N 
X_k>
\frac{x}{\epsilon_N}) 
$$
so that
\begin{equation}
\lim_{N\ra\infty}\P(|S_{N,\epsilon_N}|> x) = 0
\end{equation}
These two facts show that $S_N= \frac{1}{\widetilde a_N} \sum_{k=1}^{N}
G_{N,k}X_k$ converge in distribution to a
$Stable_{\alpha}(\sigma_{\alpha},\beta_{\alpha},0)$ variable as
announced in Theorem \ref{real stable laws}.
\end{proof}

This result implies easily the following analogous result in the
complex case.

\begin{theo}\label{complex stable laws}
Assume that the triangular array of complex numbers $(G_{N,k}, N
\geq 1, 1\le k\le N)$ is bounded. Furthermore assume that the
empirical measure
$$
  \nu_N= \frac{1}{N} \sum_{k=1}^{N} \delta_{G_{N,k}}
  \label{eqno1}
$$
converges weakly to a probability measure $\nu $ on the complex plane. Then $S_N= \frac{1}{\widetilde a_N} \sum_{k=1}^{N} G_{N,k}X_k$ converges in
distribution to a complex stable distribution with spectral
representation $(\Gamma_{\nu}, 0)$ where $\Gamma_{\nu}$ is the
measure on $S^1$ obtained as the image of the measure
$\frac{1}{C_{\alpha}}|z|^{\alpha}\nu(dz)$ on the complex plane by
the map $ z \rightarrow \frac{z}{|z|}$. Again if $ \nu= \delta_0$
the above statement should be understood as: $S_N$ converges in
distribution to zero.
\end{theo}

\begin{proof}
For any fixed $ t \in \C$, a direct application of Theorem \ref{real
stable laws} to the array of real numbers $(\langle
t,G_{N,k}\rangle)$  shows that $\langle t,S_N\rangle$ converges in
distribution to a $Stable_{\alpha}(\sigma(t),\beta(t),0)$ variable,
where

$$
  \sigma(t)^{\alpha}=\frac{1}{C_{\alpha}}\int|\langle
t,z\rangle|^{\alpha}d\nu(z)
$$
and
$$
  \beta(t)=\frac{\int|\langle t,z\rangle|^{\alpha} sign\langle
t,z\rangle d\nu(z)}{\int|\langle
  t,z\rangle|^{\alpha}d\nu(z)}.
$$
As a consequence, we obtain that 
$$
\lim_{N \ra \infty} \E[\exp(i \langle t,S_N\rangle
)]=\exp{[\sigma(t)^{\alpha}(1 - i
\beta(t)\tan(\frac{\pi\alpha}{2}))]}
$$
Note that, by definition of $\Gamma_{\nu}$:
$$
\sigma(t)^{\alpha}(1 - i
\beta(t)\tan(\frac{\pi\alpha}{2}))=\int_{S^1}| \langle t,s\rangle
|^{\alpha}(1 - i sign(\langle t,s\rangle)
\tan(\frac{\pi\alpha}{2}))\Gamma_{\nu}(ds)
$$
These two last facts prove that the distribution of $S_N$ converges
to a complex $\alpha$-stable distribution with spectral
representation $(\Gamma_{\nu},0)$.

\end{proof}

In Section \ref{limitingeq} we need a slight variation of Theorem
\ref{complex stable laws}. We want to extend it to the case where
the random variables $X_k$ are truncated at a high enough level.
More precisely, keeping the notations and hypothesis of Theorem
\ref{complex stable laws}, we define, for any $\delta>0$, the
truncated variables
$$
X^{\delta}_k= X_k 1_{X_k\le N^{\delta}\widetilde a_N}
$$
We then consider the normalized sum
$$
S^{\delta}_N= \frac{1}{\widetilde a_N} \sum_{k=1}^{N} G_{N,k}X^{\delta}_k
$$

\begin{theo}\label{extended complex stable laws}
Assume that the triangular array of complex numbers $(G_{N,k}, N
\geq 1, 1\le k\le N)$ is bounded. Furthermore assume that the
empirical measure
\begin{equation}
  \nu_N= \frac{1}{N} \sum_{k=1}^{N} \delta_{G_{N,k}}
  \label{eqno1}
\end{equation}
converges weakly to a probability measure $\nu $on the complex plane
. Then $S^{\delta}_N$ converges in distribution to a complex stable
distribution with spectral representation $(\Gamma_{\nu}, 0)$ where
$\Gamma_{\nu}$ is the measure on $S^1$ obtained as the image of the
measure $\frac{1}{C_{\alpha}}|z|^{\alpha}\nu(dz)$ on the complex
plane by the map $ z \rightarrow \frac{z}{|z|}$. Again if $ \nu=
\delta_0$ the above statement should be understood as: $S_N$
converges in distribution to zero.
\end{theo}
The proof of this variant is identical verbatim to the proof of
Theorem \ref{complex stable laws}, we omit it.

\bigskip

Finally we also need in Section \ref{limitingeq} an information
about the Fourier-Laplace transform of certain complex stable
distributions. Consider a probability measure $\nu$ on $\C$ and
define as above the measure $\Gamma_{\nu}$. Let us denote by
$P^{\nu}$ the complex $Stable_{\alpha}(\Gamma_{\nu},0)$
distribution.

\begin{theo}\label{FL transform of complex stable laws}
Assume that the measure $\nu$ is compactly supported in the closure
of $\C^-$. Then, for any $t>0$:
\begin{equation}\label{cocott}
\int e^{-itx} dP^{\nu}(x)= \exp (-\Gamma(1-\alpha)(it)^{\alpha}\int
x^{\alpha}d\nu(x))
\end{equation}

\end{theo}

\begin{proof}
This is a simple consequence of the analogous result for real
$Stable_{\alpha}(\sigma,1)$ distributions. If $X$ is a random
variable with $Stable_{\alpha}(\sigma,1)$ distribution, and if
$\gamma \in \C$ with $\Re(\gamma)>0$, then
\begin{equation}\label{FL transform for real stable laws}
\E(e^{-\gamma X})=
e^{-\frac{\sigma^{\alpha}}{\cos(\frac{\pi\alpha}{2})}\gamma^{\alpha}}
\end{equation}
This result is classical when $\gamma$ is real positive (see
Proposition 1.2.12 of \cite{Taqqu} for instance). The statement
(\ref{FL transform for real stable laws}) is obtained by an easy
analytic extension from the real case.

Consider now a sequence of i.i.d.r.v $(X_k)_{k \geq 1}$, with common
distribution $Stable_{\alpha}(\sigma,1)$. Furthermore consider a
bounded array of complex numbers $(G_{N,k}) \in \C^-$, such that the
empirical measure $ \frac{1}{N}\sum_{k=1}^{N} \delta_{G_{N,k}}$
converges to $\nu$ when $N \ra \infty$. As above define the
normalized sum $$S_N=\frac{1}{\widetilde a_N} \sum_{k=1}^{N}G_{N,k}X_k$$ Then,
if $ \gamma_{N,k} = it \frac{G_{N,k}}{a_N}$, one has obviously
$$
\E(e^{-itS_N})= \prod_{k=1}^{N}\E(\exp(-\gamma_{N,k} X_k))
$$
Noting that $\Re(\gamma_{N,k})>0$, it is then possible to use
(\ref{FL transform for real stable laws}):
$$
\E(e^{-itS_N})=
\exp(-\frac{\sigma^{\alpha}}{\cos(\frac{\pi\alpha}{2})}
\sum_{k=1}^{N}\gamma_{k,N}^{\alpha} )
$$
Using the classical tail estimate for real
$Stable_{\alpha}(\sigma,1)$ distributions, when u tends to $\infty$:
$$
P(X \geq u) \sim \frac{C_{\alpha}\sigma^{\alpha}}{u^{\alpha}}
$$
one sees that $ \widetilde a_N \sim C_{\alpha}^\frac{1}{\alpha}
N^{\frac{1}{\alpha}}$.

Thus, we get the estimate
$$
\E(e^{-itS_N})\sim
\exp(-\frac{(it)^{\alpha}}{C_{\alpha}\cos(\frac{\pi\alpha}{2})}
\frac{1}{N}\sum_{k=1}^{N}G_{N,k}^{\alpha} ).
$$

But $\frac{1}{N}\sum_{k=1}^{N}G_{N,k}^{\alpha}$ converges to $\int
x^{\alpha}d\nu(x)$. Using now the convergence theorem \ref{complex
stable laws} we see that,
$$
\int e^{-itx} dP^{\nu}(x)= \lim_{N \ra \infty} \E(e^{-itS_N})= \exp
(-\frac{1}{C_{\alpha}\cos(\frac{\pi\alpha}{2})}(it)^{\alpha}\int
x^{\alpha}d\nu(x)).
$$
Noting that
$$
C_{\alpha}\cos(\frac{\pi\alpha}{2})=
\frac{1-\alpha}{\Gamma(2-\alpha)}= \frac{1}{\Gamma(1-\alpha)}
$$
 proves Theorem \ref{FL transform of complex stable laws}.
\end{proof}

{\bf Acknowledgments.} The authors wish to thank A. Soshnikov for
pointing out this problem to them during a conference at Banff in
2004. They are very grateful to S. Belinschi for very  useful
comments.


\begin{thebibliography}{10}

\bibitem{bai}
{\sc Bai, Z.~D.}
\newblock Methodologies in spectral analysis of large-dimensional random
  matrices, a review.
\newblock {\em Statist. Sinica 9}, 3 (1999), 611--677.
\newblock With comments by G. J.\ Rodgers and Jack W.\ Silverstein; and a
  rejoinder by the author.

\bibitem{Bingham}
{\sc Bingham, N.~H., Goldie, C.~M., and Teugels, J.~L.}
\newblock {\em Regular variation}, vol.~27 of {\em Encyclopedia of Mathematics
  and its Applications}.
\newblock Cambridge University Press, Cambridge, 1989.

\bibitem{CB}
{\sc Bouchaud, J., and Cizeau, P.}
\newblock Theory of {L}\'evy matrices.
\newblock {\em Phys. Rev. E}, 3 (1994), 1810--1822.

\bibitem{burda}
{\sc Burda, Z., Jurkiewicz, J., Nowak, M., and Zahed, I.}
\newblock Random {L}\'evy matrices revisited.
\newblock {\em arxiv/cond-mat/0602087\/} (2006).

\bibitem{collingwood}
{\sc Collingwood, E.~F., and Lohwater, A.~J.}
\newblock {\em The theory of cluster sets}.
\newblock Cambridge Tracts in Mathematics and Mathematical Physics, No. 56.
  Cambridge University Press, Cambridge, 1966.

\bibitem{Feller}
{\sc Feller, W.}
\newblock {\em An introduction to probability theory and its applications.
  {V}ol. {II}.}
\newblock Second edition. John Wiley \& Sons Inc., New York, 1971.

\bibitem{Galambos}
{\sc Galambos, J.}
\newblock {\em Advanced probability theory}, second~ed., vol.~10 of {\em
  Probability: Pure and Applied}.
\newblock Marcel Dekker Inc., New York, 1995.

\bibitem{GZe}
{\sc Guionnet, A., and Zeitouni, O.}
\newblock Large deviations asymptotics for spherical integrals.
\newblock {\em J. Funct. Anal. 188}, 2 (2002), 461--515.

\bibitem{pastur}
{\sc Khorunzhy, A.~M., Khoruzhenko, B.~A., and Pastur, L.~A.}
\newblock Asymptotic properties of large random matrices with independent
  entries.
\newblock {\em J. Math. Phys. 37}, 10 (1996), 5033--5060.

\bibitem{Pommerenke}
{\sc Pommerenke, C.}
\newblock {\em Boundary behaviour of conformal maps}, vol.~299 of {\em
  Grundlehren der Mathematischen Wissenschaften [Fundamental Principles of
  Mathematical Sciences]}.
\newblock Springer-Verlag, Berlin, 1992.

\bibitem{Taqqu}
{\sc Samorodnitsky, G., and Taqqu, M.~S.}
\newblock {\em Stable non-{G}aussian random processes}.
\newblock Stochastic Modeling. Chapman \& Hall, New York, 1994.
\newblock Stochastic models with infinite variance.

\bibitem{sosh}
{\sc Soshnikov, A.}
\newblock Universality at the edge of the spectrum in {W}igner random matrices.
\newblock {\em Comm. Math. Phys. 207}, 3 (1999), 697--733.

\bibitem{soshi}
{\sc Soshnikov, A.}
\newblock Poisson statistics for the largest eigenvalues in random matrix
  ensembles.
\newblock In {\em Mathematical physics of quantum mechanics}, vol.~690 of {\em
  Lecture Notes in Phys.} Springer, Berlin, 2006, pp.~351--364.

\bibitem{wigner}
{\sc Wigner, E.~P.}
\newblock On the distribution of the roots of certain symmetric matrices.
\newblock {\em Ann. of Math. (2) 67\/} (1958), 325--327.

\bibitem{zakh}
{\sc Zakharevich, I.}
\newblock A generalization of {W}igner's law.

\end{thebibliography}

\end{document}